\documentclass[11pt,reqno]{amsart}

\usepackage{mathtools}
\usepackage{amssymb}
\usepackage{mathdots}

\usepackage{enumitem}

\usepackage{graphicx}
\usepackage{tikz-cd}

\usepackage{xspace}
\usepackage{xcolor}

\usepackage[
  colorlinks=true,
  linkcolor=blue,
  citecolor=blue,
  urlcolor=blue
]{hyperref}

\theoremstyle{plain}
\newtheorem{theorem}{Theorem}[section]
\newtheorem{proposition}[theorem]{Proposition}
\newtheorem{lemma}[theorem]{Lemma}
\newtheorem{corollary}[theorem]{Corollary}
\newtheorem{example}[theorem]{Example}

\theoremstyle{definition}
\newtheorem{definition}[theorem]{Definition}

\theoremstyle{remark}
\newtheorem{remark}[theorem]{Remark}

\newenvironment{demo}
  {\begin{proof}}
  {\end{proof}}

\newlist{arabiclist}{enumerate}{1}
\setlist[arabiclist]{
  label=\arabic*.,
  ref=\arabic*,
  leftmargin=*,
  itemsep=0.2em,
  topsep=0.2em
}

\newlist{alphlist}{enumerate}{1}
\setlist[alphlist]{
  label=(\alph*),
  ref=\alph*,
  leftmargin=*,
  itemsep=0.2em,
  topsep=0.2em
}

\newcommand{\Z}{\mathbb{Z}}

\newcommand{\Q}{\mathbb{Q}}

\newcommand{\R}{\mathbb{R}}

\newcommand{\norm}{\mathrm{N}}

\renewcommand{\vec}{\mathbf}

\DeclareMathOperator{\tr}{Tr}

\newcommand{\M}{\mathcal{M}}

\renewcommand{\O}{\mathcal{O}}

\renewcommand{\L}{\mathcal{L}}

\renewcommand{\P}{\mathfrak{P}}
\newcommand{\Pp}{\mathfrak{P}}
\newcommand{\p}{\mathfrak p}

\newcommand{\Qq}{\mathfrak Q}

\newcommand*{\ida}{\mathfrak a}
\newcommand*{\idb}{\mathfrak b}

\newcommand*{\idd}{\mathfrak d}
\newcommand*{\idw}{\mathfrak w}
\newcommand*{\idg}{\mathfrak g}

\newcommand*{\s}{\mathfrak s}
\newcommand*{\n}{\mathfrak n}

\renewcommand*{\H}{\mathbb H}

\newcommand*{\lm}{\mathcal M}
\newcommand*{\lk}{\mathcal K}
\newcommand*{\lf}{\mathcal F}

\newcommand*{\hf}{\mathcal M}

\newcommand*{\hk}{\mathcal F}

\newcommand{\GL}{\mathbf{GL}}

\newcommand{\disc}{\mathrm{disc}}

\newcommand{\Hawk}{\textsf{Hawk}}

\newcommand{\bigperp}
  {\mathop{\vcenter{\hbox{\scalebox{2}{$\perp$}}}}\limits}

\newcommand{\midperp}
  {\mathop{\vcenter{\hbox{\scalebox{1.5}{$\perp$}}}}\limits}

\begin{document}

\title[Explicit Jordan decompositions for ideal lattices in CM fields]{Explicit Jordan decompositions for \\ ideal lattices in CM fields}

\author{Guilhem Mureau}
\address{
  Univ Bordeaux, CNRS, Inria, Bordeaux INP, IMB, UMR 5251,
  Talence, France
}
\email{guilhem.mureau@math.u-bordeaux.fr}

\begin{abstract}
Let $E$ be a CM number field, $F$ a totally real subfield (\textit{e.g.} $\Q$), and let $\ida$
be a non-zero fractional ideal of $E$. Endowed with the hermitian trace form
$h_{E/F}(x,y)=\tr_{E/F}(x\overline y)$, the ideal $\ida$ defines an ideal
lattice over $\O_F$. In this paper, we give explicit formulas
for the Jordan decomposition of this lattice at a prime ideal
$\p \subset \O_F$ in terms of the prime ideal factorization of $\ida$ in $E$. Following the approach of Erez, Morales and Perlis, we reduce the computation to the local behavior at the prime ideals
above $\p$. Our
results provide local invariants for the isometry relation between
ideal lattices, with potential applications to the study of structured lattices
arising in arithmetic and cryptography.
\end{abstract}

\maketitle


\section{Introduction}

The question of having closed formulas for Jordan decompositions has already been investigated, in the following context. Let \(K\) be a number field with
ring of integers \(\O_K\), and let $q_K(x)=\tr_{K/\Q}(x^2)$
be the integral trace form on \(K\). Then, the genus of the quadratic lattice
\((\O_K,q_K)\) has been studied in
\cite{erez1987genre,mantillagenus,maurer1973trace}. In particular, when a
rational prime \(p\) is not wildly ramified in \(K\), a Jordan decomposition of
\((\O_K,q_K)\) at \(p\) can be described explicitly in terms of the
ramification data of \(K/\Q\) at \(p\), see for instance
\cite[Theorem~0.1]{mantillagenus}. In these works, the field varies and the
resulting decompositions are used to produce arithmetic invariants of the field
itself.

In this paper we take a different point of view: the ambient field is fixed, and the
lattice varies through the fractional ideals of the field. More precisely, we consider a totally real number field \(K\), a CM extension $E$ of $K$
and a subfield \(F\subseteq K\). The extension \(E/F\) is then endowed with the trace form $h_{E/F} : (x,y) \mapsto \tr_{E/F}(x\overline y)$,\footnote{Although the notation \(h_{E/F}\) may suggest a Hermitian form, it denotes here
a quadratic form on \(F\). The letter \(h\) is used only
to emphasize that the form is built from the expression \(x\overline y\), in
contrast with the usual trace form \((x,y)\mapsto \tr(xy)\).} and every nonzero fractional ideal \(\ida\subset E\) gives rise to the ideal
lattice
\begin{equation}\label{eq:ideal-lattice}
\L_{\ida,F}=(\ida,h_{E/F}).
\end{equation}

This setting is motivated in part by the lattice isomorphism problem for
module lattices over CM fields, which has recently attracted considerable
attention in lattice-based cryptography; see \Hawk{} \cite{ducas2022hawk}. Although the relevant cryptographic instances
typically have higher module rank, ideal lattices already form a rich and
tractable class. They therefore provide a natural testing ground for
understanding which arithmetic information is captured by local quadratic
invariants.

Our main goal is to give closed formulas for the Jordan decomposition of \(\L_{\ida,F}\) at a prime ideal \(\p\subset\O_F\), in terms of the prime ideal factorization
of \(\ida\) in $E$. In other words, we want to read local invariants of the quadratic lattice $\L_{\ida, F}$ directly from the valuations of \(\ida\) at the prime ideals of
\(\O_E\) above \(\p\). When \(F=\Q\), this gives an explicit description of the
rational genera of these ideal lattices. More generally, varying the intermediate
field \(F\subseteq K\) gives a hierarchy of local invariants: larger choices of
\(F\) usually lead to finer decompositions.

The results take different forms depending on whether \(\p\) is dyadic, that is, whether $2$ belongs to $\p$ or not. In the
non-dyadic case, and under technical assumptions on the ideal
\(\ida\), we obtain a particularly simple description of the Jordan blocks in
terms of the valuations of $\ida$ at primes dividing $\p$. When the
extension \(K/F\) is Galois, this yields an explicit local isometry criterion. In the dyadic case, the structure is more
subtle, but a similar description can still be obtained under suitable
assumptions. The following informal statement summarizes our main result; precise
versions are given in Theorem~\ref{thm:jordan-decomp-nondyadic} and
Corollary~\ref{cor:local-criterion-nondyadic} for the non-dyadic case, and in
Theorem~\ref{thm:dyadic-jordan} and
Corollary~\ref{cor:dyadic-p-separated} for the dyadic case.

\medskip
\noindent\textbf{Theorem (informal).}
\textit{Let \(\L_{\ida,F}\) be an ideal lattice, as defined in (\ref{eq:ideal-lattice}). Assuming $K/F$ is Galois, and under explicit assumptions on the prime ideal factorization of $\ida$, the Jordan decomposition of
\(\L_{\ida,F}\) at \(\p\) can be read off from the valuations of \(\ida\) at the
prime ideals \(\P\mid\p\). In particular, the local isometry
class of \(\L_{\ida,F}\) is determined by explicit combinatorial data attached
to the prime ideal factorization of \(\ida\).}

\medskip

These decompositions serve both algorithmic and arithmetic purposes. They
provide effective local invariants for comparing ideal lattices, and may also
serve as a starting point for studying their spinor genera and local densities,
in the spirit of the work of Mantilla-Soler on integral trace forms
\cite{mantilla2017spinor} and of Fiori on local density computations
\cite{fiori2013questions}.

\subsection*{Organization of the paper}

The paper is organized as follows. Section~\ref{sec:lattices} recalls the local theory
of quadratic lattices used throughout the paper. Section~\ref{sec:outline} fixes the
notation and reduces the computation of Jordan decompositions to a
collection of local problems. The non-dyadic and dyadic cases are
treated in Sections~\ref{sec:non-dyadic} and~\ref{sec:dyadic}, respectively.


\section{Local preliminaries}\label{sec:lattices}


Let \(R\) be a Dedekind domain of characteristic zero, with fraction
field \(F\). We briefly recall the terminology and properties of
quadratic \(R\)-lattices used throughout the paper, with particular
emphasis on Jordan decompositions when \(R\) is a valuation ring.
We refer to O'Meara~\cite{o2013introduction} for the general theory.
\begin{definition}[\(R\)-lattices]
Let \(V\) be a finite-dimensional \(F\)-vector space. An \(R\)-lattice
in \(V\) is a finitely generated \(R\)-submodule \(\L\subset V\) such
that \(F\otimes_R \L=V\). 

A quadratic \(R\)-lattice is a pair \((\L,q)\), where $\L$ is an $R$-lattice in $V$ and $q : V \times V \to F$ is a non-degenerate symmetric bilinear form.
\end{definition}
When the form is clear from the context, we simply write \(\L\) instead of $(\L,q)$. For a quadratic \(R\)-lattice \((\L,q)\), its dual is the $R$-lattice given by
\[
\L^\sharp
=
\{x\in \L\otimes_R F\, | \,q(x,\L)\subseteq R\}.
\]
The scale and norm of $(\L,q)$ are respectively the fractional ideals of $R$:
\[
\mathfrak s(\L)
=
\langle q(x,y)\, | \,x,y\in \L\rangle_R,
\qquad \text{and} \qquad
\mathfrak n(\L)
=
\langle q(x,x)\, | \,x\in \L\rangle_R.
\]
For a fractional ideal \(\mathfrak c\) of \(R\), the lattice \(\L\)
is called \(\mathfrak c\)-modular if
\(\L=\mathfrak cL^\sharp\).
When \(\L\) is free, \(\disc(\L)\) denotes the determinant of a Gram matrix, viewed in \(F^\times/R^{\times 2}\).

Isometries between $R$-lattices are understood to be \(R\)-isometries and are denoted by
\(\L\simeq_R \M\). We use the following notation for isometry classes:
\begin{arabiclist}
\item
For any \(a_1,\ldots,a_m\in F^\times\), we denote by \(\langle a_1,\ldots,a_m\rangle\) the quadratic
\(R\)-lattice with Gram matrix
\(\mathrm{diag}(a_1,\ldots,a_m)\);

\item
\(\H\) denotes the hyperbolic plane, with Gram matrix $\begin{psmallmatrix}
0&1\\
1&0
\end{psmallmatrix};$

\item
For \(a,b\in F\) with $ab \neq 1$, \(A(a,b)\) denotes the lattice with
Gram matrix $\begin{psmallmatrix}
a&1\\
1&b
\end{psmallmatrix}.$
\end{arabiclist}
In particular, \(A(0,0)=\H\). For \(\ell\geq 1\), the notation
\(A(a,b)^\ell\), respectively \(\H^\ell\), denotes an orthogonal
sum of \(\ell\) copies of \(A(a,b)\), respectively \(\H\).

The tensor product $\L_1 \otimes_R \L_2$ of two lattices \((\L_1,q_1)\) and \((\L_2,q_2)\) is endowed with the bilinear form determined by
\[
(q_1\otimes q_2)
(x_1\otimes x_2,y_1\otimes y_2)
=
q_1(x_1,y_1)q_2(x_2,y_2).
\]
In particular, we have: 
\[
\langle a_1,\ldots,a_s\rangle
\otimes
\langle b_1,\ldots,b_t\rangle
\simeq_R
\langle a_ib_j\, | \, 1\leq i\leq s,\ 1\leq j\leq t\rangle.
\]

\subsection{Jordan decompositions}\label{subsec:jordan-decomp}

From now on, assume that \(R\) is the valuation ring of a
non-Archimedean local field \(F\), with maximal ideal
\(\p=\pi R\). All the results stated here can be found in O'Meara~\cite[\S 91--93]{o2013introduction}.

\begin{definition}[Jordan decomposition]
Let \(\L\) be a quadratic \(R\)-lattice. A \emph{Jordan decomposition} of
\(\L\) is an orthogonal decomposition
\[
\L=\bigperp_{i=1}^t \L_i
\]
such that \(\L_i\) is \(\p^{s_i}\)-modular for every \(i\), where $s_1<\cdots<s_t.$
The lattices \(\L_i\) are called the \emph{Jordan components} of the
decomposition.
\end{definition}

Every quadratic \(R\)-lattice admits a Jordan decomposition.
Although the decomposition is not unique, the scales and ranks of
its Jordan components are independent of the chosen decomposition.

Our computations will often produce finer orthogonal decompositions
into explicit modular sublattices, which we call \emph{Jordan blocks}.
Blocks with the same scale are grouped together to form a Jordan
component.

When \(F\) is non-dyadic, the description is particularly simple. Every \(\p^s\)-modular quadratic lattice \(\M\) of rank \(m\) is
diagonalizable, and
\begin{equation}
\label{eq:nondyadic-modular-classification}
\M
 \simeq_R
 \langle \pi^s\rangle
 \otimes
 \left\langle
 1,\ldots,1,
 \disc(\M)\pi^{-sm}
 \right\rangle.
\end{equation}
Indeed, \(\disc(\M)\pi^{-sm} \in R^\times \) is a unit, and the isometry class of \(\M\) is determined by its rank, its scale and its discriminant in
\(F^\times/F^{\times 2}\). Consequently,~if
\[
\L=\L_1\perp\cdots\perp \L_t
\qquad\text{and}\qquad
\M=\M_1\perp\cdots\perp \M_t
\]
are Jordan decompositions over a non-dyadic field, with \(\L_i\) and
\(\M_i\) of the same scale \(\p^{s_i}\), then $\L\simeq_R \M$
if and only if, for every \(i\),
\[
\mathrm{rank}(\L_i)=\mathrm{rank}(\M_i)
\qquad\text{and}\qquad
\disc(\L_i) / \disc(\M_i) \in R^{\times 2}.
\]

Over dyadic fields, the scales, ranks and discriminants of the Jordan
components do not suffice to determine the isometry class. We therefore
recall the additional invariants needed below. 

\subsection{Dyadic invariants}
\label{subsec:dyadic-prelim}

Assume furthermore that \(F\) is dyadic, that is, \(2\in\p\). Following \cite[\S 63]{o2013introduction}, we recall the additional invariants that will be used later. For \(y\in F\), its quadratic defect is
\[
\mathfrak d(y)=\bigcap_{x\in F}(y-x^2)R.
\]
It satisfies \(\mathfrak d(y)=0\) if and only if \(y\) is a square, and
\(\mathfrak d(z^2y)=z^2 \mathfrak d(y)\) for every \(z\in F\); see
O'Meara~\cite[\S 63]{o2013introduction}. There exists a unit \(\Delta\in R^\times\) such that $\idd(\Delta)=4R$, and we may write $\Delta=1-4\rho$ for some unit \(\rho\in R^\times\). The following characterization will be useful later on.

\begin{lemma}[{\cite[63:3]{o2013introduction}}]
\label{lem:charac-Delta}
Let \(\Delta\in R^\times\). The following are equivalent:
\begin{arabiclist}
\item \(\idd(\Delta)=4R\);
\item \(F(\sqrt{\Delta})/F\) is the unique unramified quadratic extension of
\(F\).
\end{arabiclist}
\end{lemma}

Let \((\L,q)\) be a quadratic lattice over \(R\). The classification of
quadratic lattices over \(R\) is more subtle in the dyadic case, essentially
because \(2\) is no longer a unit. In particular, the inclusions
$2\s(\L)\subseteq \n(\L)\subseteq \s(\L)$
are in general strict. We therefore introduce two additional invariants: the
\textit{norm group} and the \textit{weight} of $(\L,q)$.

\begin{definition}[Dyadic invariants]
\begin{arabiclist}
\item The norm group of $(\L,q)$ is
\(\idg(\L) := \{q(x,x)\, | \,x\in \L\} + 2\mathfrak{s}(\L).\)
It is an additive subgroup of $F$. A \textit{norm generator} is an element $g \in \idg(\L)$ such that
$gR = \mathfrak{n}(\L)$.

\item Since $\idg(\L) \supseteq 2\mathfrak{s}(\L)$, there
exists a minimal integer $\ell \in \Z_{\ge 0}$ such that
$\p^\ell \subseteq \idg(\L)$. The \textit{weight} of $(\L,q)$ is then the ideal defined by
\[
\idw(\L) := 2\mathfrak{s}(\L) + \p^{\ell+1}.
\]
\end{arabiclist}
\end{definition}

The norm group and the weight are among the additional invariants
entering the classification of quadratic lattices over dyadic fields;
see O'Meara~\cite[\S 93]{o2013introduction}. We isolate the following rank-two case, which will be used later.

\begin{lemma}
\label{lem:dyadic-rank-two}
Let \(F\) be a dyadic local field with valuation ring \(R\) and maximal ideal
\(\p=\pi R\). Let \((\L,q)\) be a rank-two \(R\)-lattice. Assume that \(\L\) is
\(\p^v\)-modular for some \(v\in\Z\), and that the quadratic space
\((\L\otimes_R F,q)\) is anisotropic. Then
\[
(\L,q)\simeq_R \langle \pi^v\rangle\otimes A(2,2\rho),
\]
where \(\rho\in R^\times\) is such that \(\idd(1-4\rho)=4R\).
\end{lemma}

\begin{proof}
The rescaled lattice \((\L,\pi^{-v}q)\) is unimodular by
\cite[§82J]{o2013introduction}. Since its ambient quadratic space is
anisotropic, \cite[93:11]{o2013introduction} gives $(\L,\pi^{-v}q)\simeq_R A(2,2\rho)$
for some \(\rho\in R^\times\) satisfying \(\idd(1-4\rho)=4R\). Rescaling by \(\pi^v\) gives the result.
\end{proof}

\subsection{Integral bases in quadratic local extensions}

Let \(\mathcal E/\mathcal K\) be a quadratic extension of
non-Archimedean local fields, with valuation rings
\(\O_{\mathcal E}\) and \(\O_{\mathcal K}\), respectively. Let
\(\P\) be the maximal ideal of \(\O_{\mathcal K}\), and let
\(\Pi\) be a uniformizer of \(\mathcal K\). Choose
\(a\in\O_{\mathcal K}\) such that
\[
\mathcal E=\mathcal K(\sqrt{-a})
\qquad\text{and}\qquad
v_{\P}(a)\in\{0,1\}.
\]
We will repeatedly use the following explicit description of
\(\O_{\mathcal E}\), due to~Fröhlich.

\begin{lemma}[{\cite[Theorem 3.1]{frohlich1960discriminants}}]\label{lem}
Let \(\delta\) denote the relative
discriminant ideal.
\begin{arabiclist}
\item If $2 \notin \Pp$, then $\delta = (a)$ and $\O_{\mathcal{E}} = \O_{\lk}[\sqrt{-a}]$.
\item Suppose that $2 \in \Pp$. 
\begin{alphlist}
\item If $a \in \Pp$, then $\delta = (4a)$ and $\O_{\mathcal{E}} = \O_{\lk}[\sqrt{-a}]$;
\item If $a \notin \Pp$, let $s \geq 0$ be the greatest integer such that
$$4 \in \Pp^{2s} \quad \text{and} \quad  \exists \, \beta_\Pp \in \O_{\lk}^\times : \beta_\Pp^2 \equiv -a \pmod{\Pp^{2s}}.$$
Then, $\delta
=
\bigl((2\Pi^{-s})^2a\bigr),$ and $\O_{\mathcal{E}} = \O_{\lk}\Big[\frac{\beta_\Pp + \sqrt{-a}}{\Pi^s}\Big].$
\end{alphlist}
\end{arabiclist}
\end{lemma}

\begin{remark}\label{rmk}
The extreme cases \(s=0\) and \(s=v_{\P}(2)\) correspond respectively
to \(\O_{\mathcal E}=\O_{\mathcal K}[\sqrt{-a}]\) and to an
unramified extension.
\end{remark}


\section{Outline of the method}\label{sec:outline}

The overall strategy is directly adapted from
\cite{erez1987genre,mantillagenus}. We recall it here, after fixing the notation used throughout this paper. 

\subsection{Notation and local conventions}\label{sec:notation-trace}

Let \(K\) be a totally real number field, and choose
\(a\in\O_K\) satisfying:
\begin{arabiclist}
\item \(a\) is totally positive, that is,
\(\sigma(a)>0\) for every embedding \(\sigma\colon K\hookrightarrow\R\);
\item the ideal \(a\O_K\) is square-free.
\end{arabiclist}
Since \(a\) is totally positive, the polynomial \(X^2+a\) is
irreducible over \(K\). We set $E=K[X]/(X^2+a)$
and identify \(\sqrt{-a}\) with the image of \(X\) in \(E\). Then
\(E/K\) is a quadratic CM extension. Its non-trivial
\(K\)-automorphism is complex conjugation, given by $\overline{\phantom{x}}\colon
x+y\sqrt{-a}\longmapsto x-y\sqrt{-a}.$

Let \(F\subseteq K\) be a fixed subfield (\textit{e.g.,} \(F = \Q\)). The map
\[
\begin{aligned}
h_{E/F}\colon E\times E&\longrightarrow F,\\
(x,y)&\longmapsto \tr_{E/F}(x\overline y)
\end{aligned}
\]
is a non-degenerate symmetric \(F\)-bilinear form, and hence turns
\(E\) into a quadratic space over \(F\). Consequently, every non-zero
fractional \(\O_E\)-ideal \(\ida\) defines the quadratic
\(\O_F\)-lattice
\begin{equation}
\label{eq:ideal-lattice}
\L_{\ida,F}=(\ida,h_{E/F}).
\end{equation}

Finally, fix a non-zero prime ideal \(\p\subset\O_F\), and assume
that \(K/F\) is tamely ramified at \(\p\). We write $\p\O_K=\prod_{i=1}^g\P_i^{e_i}$
for its prime ideal factorization in \(\O_K\). Equivalently, each
ramification index \(e_i\) satisfies $e_i \notin \p$.
\\ \ \\
\textit{Notation.} Throughout the paper, global fields are denoted by italic letters,
such as \(K\), \(E\), and \(F\), whereas calligraphic letters, such as
\(\lk\) and \(\lf\), are used for local fields. We use \(\L\) or $\M$ for
lattices and Fraktur letters for ideals.

\subsection{Decomposition over the primes above \(\p\)}
\label{subsec:descent}

For each prime ideal \(\P\mid\p\), let \(F_\p\) and \(K_\P\)
denote the corresponding completions. The standard decomposition of completions gives
\begin{equation}
\label{eq:prod-completion}
K\otimes_F F_{\p}
 \simeq
\prod_{\P\mid\p}K_{\P}
\qquad \text{and} \qquad
\O_K\otimes_{\O_F}\O_{F_{\p}}
 \simeq
\prod_{\P\mid\p}\O_{K_{\P}},
\end{equation}
see, for instance,
\cite[Chapter~II, Corollary~8.4]{neukirch2013algebraic}. Moreover, set $E_{\p}=E\otimes_F F_{\p}$ and $E_{\P}=E\otimes_K K_{\P}$.
Then \eqref{eq:prod-completion} induces an isomorphism
\[
E_{\p}\simeq\prod_{\P\mid\p}E_{\P}.
\]
Complex conjugation acts componentwise under this identification, and
the trace form on \(E_{\p}\) is the orthogonal sum of the trace forms
on the factors \(E_{\P}\). At the level of fractional ideals, we prove the following lemma, generalizing \cite[Equation (1.5)]{erez1987genre}.

\begin{lemma}
\label{lemma:orthoplit}
Let \(\ida\) be a non-zero fractional ideal of \(E\). There is an
orthogonal decomposition of quadratic \(\O_{F_{\p}}\)-lattices
\[
\left(
\ida\otimes_{\O_F}\O_{F_{\p}},
h_{E_{\p}/F_{\p}}
\right)
\simeq_{\O_{F_{\p}}}
\bigperp_{\P\mid\p}
\left(
\ida\otimes_{\O_K}\O_{K_{\P}},
h_{E_{\P}/F_{\p}}
\right).
\]
\end{lemma}

\begin{proof}
By associativity of tensor products and \eqref{eq:prod-completion},
\[
\begin{aligned}
\ida\otimes_{\O_F}\O_{F_\p}
&\simeq
\ida\otimes_{\O_K}
\left(\O_K\otimes_{\O_F}\O_{F_\p}\right)\\
&\simeq
\bigoplus_{\P\mid\p}
\left(\ida\otimes_{\O_K}\O_{K_\P}\right).
\end{aligned}
\]
Under the decomposition $E_\p\simeq\prod_{\P\mid\p}E_\P,$
complex conjugation acts componentwise. Hence, for
\(x=(x_\P)_{\P\mid \p}\) and \(y=(y_\P)_{\P \mid \p}\), one has
\[
h_{E_\p/F_\p}(x,y)
=
\sum_{\P\mid\p}
\tr_{E_\P/F_\p}(x_\P\overline{y_\P}).
\]
The summands are therefore orthogonal, and the restriction of
the form to the \(\P\)-component is exactly \(h_{E_\P/F_\p}\).
\end{proof}

\subsection{Descent through intermediate extensions}\label{subsec:descent-through}

By Lemma~\ref{lemma:orthoplit}, it is enough to compute a splitting of
\(\ida\otimes_{\O_K}\O_{K_{\P}}\) for each prime ideal
\(\P\mid\p\). Fix such a prime \(\P\), and set $\lf=F_{\p}$ and $\lk=K_{\P}.$ Let \(\lm\) be the maximal unramified subextension of \(\lk/\lf\).
We then have the tower:
\[
\begin{tikzcd}[row sep=1.6em]
E\otimes_K\lk
  \arrow[d, no head, "2"'] \\
\lk
  \arrow[d, no head, "e"'] \\
\lm
  \arrow[d, no head, "f"'] \\
\lf
\end{tikzcd}
\qquad 
\left\{
\begin{array}{@{}l@{}}
\lk/\lm \text{ is totally tamely ramified},\\[0.2cm]
\lm/\lf \text{ is unramified},\\[0.2cm]
ef=[\lk:\lf].
\end{array}
\right.
\]
Here \(e=e(\P\mid\p)\) is the ramification index and $f=[\O_{\lk}/\P\O_{\lk}:\O_{\lf}/\p\O_{\lf}]$
is the residue degree. By transitivity of the trace,
\[
\tr_{E\otimes_K\lk/\lf}
=
\tr_{\lm/\lf}
\circ
\tr_{\lk/\lm}
\circ
\tr_{E\otimes_K\lk/\lk}.
\]
Accordingly, the computation can be split in three stages: first over
\(\lk\), then through the totally ramified extension \(\lk/\lm\), and
finally through the unramified extension \(\lm/\lf\). We treat separately the non-dyadic case \(2\notin\P\) and the dyadic
case \(2\in\P\). Sections~\ref{sec:cm}--\ref{sec:gluing-non-dyadic}
carry out the three steps in the non-dyadic setting, while
Section~\ref{sec:dyadic} deals with the dyadic case.


\section{The non-dyadic case}\label{sec:non-dyadic}

We carry on the notation introduced in Subsection \ref{sec:notation-trace}.

\subsection{Over the CM extension}\label{sec:cm} 

Let us fix a prime ideal $\Pp \subset \O_K$ above $\p$. Since $E/K$ is a quadratic extension, there are only three possibilities for the prime ideal factorization of $\Pp\O_E$.
\begin{arabiclist}
\item Split: $\Pp\O_E = \Qq_1 \Qq_2$ with $\Qq_1 \neq \Qq_2$ and $\Qq_2 = \overline{\Qq_1};$ 
\item Inert: $\Pp\O_E = \Qq$ is a prime ideal;
\item Ramified: $\Pp\O_E = \Qq^2$ is the square of a prime ideal. 
\end{arabiclist}

As in Subsection \ref{subsec:descent-through}, we set $\lk = K_\P$ and $E_\P = E \otimes_{K} \lk$. The latter is a two-dimensional vector space over $\lk$ and a quadratic space, endowed with the form 
$$ 
(x,y) \mapsto  \tr_{E_\P/\lk}(x\overline{y}),
$$
where $\tr_{E_\P/\lk}$ is defined by $\tr_{E_\P/\lk} = z + \overline{z}$ for any $z \in E_\P$. In this section our goal is to compute  splittings of $\O_{\lk}$-lattices 
\begin{equation}\label{eq:l_a}
\L_{\ida,\lk} := \Big(\ida \otimes_{\O_K} \O_{\lk} , \tr_{E_\P/\lk}(x\overline{y})\Big),
\end{equation}
where $\ida$ is any ideal of $\O_E$. Note that such lattices have rank-two over $\O_{\lk}$, and that it is enough to consider $\ida$ which is a product of prime ideals dividing $\P$. 

Recall that $E_\P := E \otimes_{K} \lk$ is isomorphic to $\prod_{\Qq | \Pp} E_{\Qq},$ where the product ranges over the prime ideals $\Qq$ of $\O_E$ dividing $\P$, and that this isomorphism preserves integral elements, \textit{i.e.},
\begin{equation}\label{eq:phi}
\phi : \O_{E} \otimes_{\O_{K}} \O_{\lk} \xrightarrow{\sim} \prod_{\Qq | \Pp} \O_{E_\Qq}
\end{equation} 
is an isomorphism of $\O_{K}$-modules, where $\O_{E_{\Qq}}$ is the valuation ring of $E_{\Qq}$. 

\subsubsection{Split case} Assume $\P$ splits, that is, $\P \O_E = \Qq_1 \Qq_2$ for distinct prime ideals $\Qq_1, \Qq_2$ of $\O_E$. Note that $E_{\Qq_1}, E_{\Qq_2}$ are degree one extensions of $\lk$, and are thus isomorphic to $\lk$. Then, Equation (\ref{eq:phi}) now reads 
\[\O_E \otimes_{\O_K} \O_{\lk} \simeq_{\O_K} \O_{\lk} \times \O_{\lk}.\] 
Moreover, complex conjugation $\overline{\, \cdot \,}$ on $\O_{E} \otimes_{\O_{K}} \O_{\lk}$ corresponds to the involution on $\O_{\lk} \times \O_{\lk}$ which swaps coordinates. 

Recall that is enough to consider ideal lattices $\L_{\ida, \lk}$ where $\ida = \Qq_1^{v_1} \Qq_2^{v_2}$. For simplicity of notation, we will denote the latter by $\L_{v_1,v_2, \lk}$. 

\begin{lemma}\label{lem:split}
Let $J_i = \Qq_i \otimes \O_{\lk}$ for $i \in \{1,2\}$. Up to exchanging $\Qq_1$ and $\Qq_2$, the isomorphism $\phi$ from (\ref{eq:phi}) maps $J_1$ on $\Pp \times \O_{\lk}$, and $J_2$ on $\O_{\lk} \times \Pp$. These are actually isometries of $\O_{\lk}$-lattices if one endows $\lk \times \lk$ with the hyperbolic form  
$$
H : ((x_1,y_1)\, ;(x_2,y_2)) \longmapsto x_1y_2 + x_2y_1.
$$
Moreover, for $v_1, v_2 \in \Z$, there is an isometry of lattices
$$
\L_{v_1,v_2,\lk} \simeq_{\O_{\lk}} (\P^{v_1} \times \P^{v_2}, H).
$$
\end{lemma}

\begin{demo}
Under $\O_E\otimes_{\O_K}\O_{\mathcal K}
\simeq
\O_{\mathcal K}\times\O_{\mathcal K},$
the two prime ideals above \(\P\) correspond, up to ordering, to
\(\P\times\O_{\mathcal K}\) and
\(\O_{\mathcal K}\times\P\). Since the isomorphism is multiplicative,
\[
\phi(\Qq_1^{v_1}\Qq_2^{v_2} \otimes \O_{\lk}) = \phi(J_1^{v_1} J_2^{v_2}) = \phi(J_1)^{v_1} \phi(J_2)^{v_2} = \P^{v_1} \times \P^{v_2}.
\]
For $(x_1, y_1), (x_2,y_2) \in \lk \times \lk$, we compute
\begin{align*}
\tr_{(\lk \times \lk)/\lk}\left((x_1,y_1) \cdot \overline{(x_2,y_2)}\right) = \tr_{(\lk \times \lk)/\lk} (x_1y_2, x_2y_1 ) 
= x_1y_2 + x_2y_1.
\end{align*}
Hence, the corresponding form is the hyperbolic form on $\lk \times \lk$. 
\end{demo}

\subsubsection{Inert or ramified case} 
Assume now that $\Pp$ is either inert or ramified, that is, $\P\O_E = \Qq$ or $\Qq^2$ where $\Qq$ is a prime ideal. In particular, (\ref{eq:phi}) now reads 
\[
\O_E \otimes_{\O_K} \O_{\lk} \simeq \O_{E_\Qq},
\] 
where $\O_{E_\Qq}$ is the valuation ring of $E_{\Qq}$. The aim of this paragraph is to compute a splitting for the $\O_{\lk}$-lattice $\L_{\ida, \lk}$, as defined in (\ref{eq:l_a}), where $\ida = \Qq^v$ and $v \in \Z$. This lattice will be denoted simply by $\L_{v, \lk}$.

\begin{lemma}\label{lem:inert-ram-nondyadic} Assume that $\Pp$ is a non-dyadic prime ideal that is either inert or ramified in $E$, and let $v \in \Z$. 
\begin{arabiclist}

	\item If $a \notin \P$, then $\P$ is inert, and $\L_{v, \lk} \simeq_{\O_{\lk}} (\P^{v}, 2x^2) \, \perp 	\, (\P^{v}, 2{a}x^2).$
	\item If $a \in \P$, then $\P$ is ramified, and $\L_{v, \lk} \simeq_{\O_{\lk}} (\P^{\lceil v/2 \rceil}, 2x^2) \, \perp \, (\P^{\lfloor v/2 \rfloor},2{a}x^2).
	$
\end{arabiclist}
\end{lemma}

\begin{demo}
By Lemma \ref{lem}, $\O_{E_\Qq} = \O_{\lk}[\sqrt{-a} \,]$ and the relative discriminant of $E_{\Qq}/\lk$ is $(a)$. In particular, $\P$ is ramified if and only if $a \in \P$. 

$1.$ If $a \notin \P$, then 
$\Qq = \P\O_{E_\Qq}$ is inert, and 
$
\Qq^v = \P^v\O_{E_\Qq} = \P^v \oplus \sqrt{-a} \, \P^v.
$

$2.$ Now if $a \in \P$, then $\Qq^2 = \P \O_{E_\Qq}$ is ramified. Moreover, $a \notin \P^2$ since $(a)$ is square-free, so that $a$ is a uniformizer of $\O_{\lk}$, and $\Qq = (\sqrt{-a})$. Moreover if $v = 2\ell$ is even, 
$\Qq^v = (\sqrt{-a}\,)^{2\ell} = \P^\ell(\O_{\lk} \oplus \sqrt{-a} \, \O_{\lk}) = \P^\ell \oplus \sqrt{-a}\, \P^\ell.$
On the other hand if $v = 2\ell + 1$ is odd, 
$\Qq^{v} = \Qq^{2\ell} \Qq = \P^{\ell} (\sqrt{-a} \, \O_{\lk} \oplus \P) = \P^{\ell+1} \oplus \sqrt{-a} \, \P^\ell.$

Since $\tr_{E_\P/\lk}(\sqrt{-a}) = 0$, the direct summands we computed above are orthogonal for $\tr_{E_\P/\lk}(x\overline{y})$. Regarding the restriction of the form, for $x \in \lk$, we have $\bar{x} = x$, so $\tr_{E_\P/\lk}(x\overline{x}) = \tr_{E_\P/\lk}(x^2) = 2x^2.$
Similarly, for $x = \sqrt{-a} \, y \in \sqrt{-a}\,\lk$, we have $x\bar{x} = a y^2$, which yields $\tr_{E_\P/\lk}(x\overline{x}) = \tr_{E_\P/\lk}(a y^2) = 2a y^2.$
\end{demo}

\subsection{Over the totally ramified extension}\label{sec:totram}

Let $K$ and $\P$ be as in the previous section. We fix a subfield $F \subseteq K$ and we set $\p = \O_F \cap \P$. Recall the notation $\lk = K_\P$ and $\lf = F_\p$ in what follows. Let $\lm$ be the maximal unramified subextension of $\lk/\lf$, and let $\O_\lm$ be its valuation ring, with maximal ideal $\p_\lm$. By assumption, \(\lk/\lm\) is totally tamely ramified of degree
\(e=[\lk:\lm]\), with \(e\notin\p_\lm\).

The structure theorem for tamely ramified extensions of local fields implies
that one may choose a uniformizer \(\Pi\) of \(\O_\lk\) such that
\(\Pi^e\in\O_\lm\) is a uniformizer
\cite[Chapter~II, \S7, Proposition~7.7]{neukirch2013algebraic}. Writing
\(\pi=\Pi^e\), and since a uniformizer generates the ring of integers in a
totally ramified extension, we obtain:
\[
\O_\lk=\O_\lm[\Pi]
\qquad \text{and} \qquad
\Pi^e=\pi \in \O_\lm \ \text{is a uniformizer}.
\]
In particular, the minimal polynomial of the extension $\lk/\lm$ is $X^e - \pi$. In particular for all $w \in \Z$, one has
\begin{equation}\label{eq:comp-trace-prelim}
\tr_{\lk/\lm}(\Pi^w) = \left\{
    \begin{array}{ll}
        0 & \mbox{if} \ e \nmid w; \\
        e\pi^q & \mbox{if} \ w = qe.
    \end{array}
\right.
\end{equation}

Recall that in the previous section, $E = K(\sqrt{-a})$ denoted a CM extension, and that we computed splittings for lattices $\L_{\ida, \lk} = (\ida \otimes \O_{\lk},\tr_{E_\P/\lk}(x\overline{y}))$. Now our goal is to bring these splittings ``down to $\M$'', that is, we view $\ida \otimes \O_{\lk}$ as a quadratic lattice over $\O_\lm$, by composing with $\tr_{\lk/\lm}$. We denote this lattice by 
$$
\L_{\ida, \lm} := \Big(\ida \otimes \O_\lk, \tr_{\lk/\lm} \circ \tr_{E_\P/\lk}(x\overline{y})\Big).
$$ 
A splitting for $\L_{\ida, \lm}$ over $\O_\lm$ is then obtained in three steps: first, we compute a Gram matrix; next, we identify the modular blocks; and finally, we apply a classification result, \textit{e.g.,} \eqref{eq:nondyadic-modular-classification} for non-dyadic primes. 

As in Section~\ref{sec:cm}, we need to distinguish the split primes from the inert or ramified ones. Let us fix some notation that will be used in both situations.

\begin{definition}\label{defi:mat-J-G}
For $q \in \Z$ and $\ell,r \in \Z_{\geq 0}$, we define matrices $J_q ^{(\ell)} \in \M_\ell(\O_\lm)$ and $\vec{G}_{q,r} \in \M_e(\O_\lm)$ by
$$
J_q^{(\ell)} = \begin{pmatrix}
 0 &  & e\pi^q\\
 & \iddots &  \\
e\pi^q &  & 0 \\
\end{pmatrix} \quad \text{and} \quad 
\vec{G}_{q,r} = \left\{
    \begin{array}{ll}
        \begin{pmatrix}
e\pi^q & \mathbf{0} \\
\mathbf{0} & J_{q+1}^{(e-1)}
\end{pmatrix} & \mbox{if} \ r = 0; 
\vspace*{0.3cm}
\\
       \begin{pmatrix}
J_{q+1}^{(e-r+1)} & \mathbf{0} \\
\mathbf{0} & J_{q+2}^{(r-1)}
\end{pmatrix} & \mbox{if} \ r > 0.
    \end{array}
\right.
$$
\end{definition}

\subsubsection{Split case} Assume $\P$ splits in $E$, \textit{i.e.}, $\P \O_E = \Qq_1 \Qq_2$ with $\Qq_2 = \overline{\Qq_1}$. For any $v_1, v_2 \in \Z$, we denote by $\L_{v_1,v_2,\lm}$ the quadratic lattice $\L_{\ida,\lm}$ with $\ida = \Qq_1^{v_1} \Qq_2^{v_2}$.

\begin{lemma}\label{lem:Gram-split} Assume that $\P$ splits in $E$, and let $v_1, v_2 \in \Z$. Denote by $qe + r$ the Euclidean division of $v_1 + v_2$ by $e$. Then $\L_{v_1,v_2, \lm}$ has Gram matrix $$ \begin{pmatrix} \vec{0} & \vec{G}_{q,r} \\ \vec{G}_{q,r} & \vec{0} \end{pmatrix}. $$ \end{lemma} \begin{demo} Recall from Lemma \ref{lem:split} that $\L_{v_1,v_2, \lk}$ is isometric to $(\P^{v_1} \times \P^{v_2}, H)$, as lattices over $\O_{\lk}$. Thus, going down to $\O_\lm$, we have to compute a Gram matrix for $\P^{v_1} \times \P^{v_2}$ with respect to the quadratic form $$ H_{\lk/\lm} : ((x_1, y_1) \, ; \, (x_2,y_2)) \longmapsto \tr_{\lk/\lm}(x_1y_2 + x_2y_1). $$ Let $\mathcal{B} := \{(\Pi^{v_1 + i}, 0)\}_{0 \leq i \leq e-1} \cup \{(0,\Pi^{v_2 + j})\}_{0 \leq j \leq e-1}$ be the canonical basis of $(\P^{v_1} \times \P^{v_2}, H_{\lk/\lm})$ over $\O_\lm$. Its Gram matrix with respect to $H_{\lk/\lm}$ is $$ \begin{pmatrix} \vec{0} & \vec{A} \\ \vec{A} & \vec{0} \end{pmatrix}, \quad \text{where} \ \vec{A} = \Big(\tr_{\lk/\lm}(\Pi^{v_1+v_2 + i + j})\Big)_{0 \leq i,j \leq e-1}. $$ For $0 \leq i,j \leq e-1$, and by (\ref{eq:comp-trace-prelim}), we have \begin{align*} \tr_{\lk/\lm}(\mathrm{\Pi}^{v_1 + v_2 + i + j}) &= \left\{ \begin{array}{ll} e\mathrm{\Pi}^{v_1 + v_2 + i + j} & \mbox{if $e \mid v_1 + v_2 + i + j$};\\ 0 & \mbox{otherwise}\end{array} \right. \\ &= \left\{ \begin{array}{ll} e\mathrm{\Pi}^{eq} = e\pi^q & \mbox{if $r = 0 $ and $i = j = 0$}; \\ e\mathrm{\Pi}^{e(q+1)} = e \pi^{q+1} & \mbox{if $i+j = e-r$}; \\ e\mathrm{\Pi}^{e(q+2)} = e \pi^{q+2} & \mbox{if $i+j = 2e-r$}; \\ 0 & \mbox{otherwise}.\end{array} \right. \end{align*} Thus if $r=0$, nonzero coefficients of $\vec{A}$ appear only for $(i,j) = (0,0)$ and $(1, e-1)$, $\dots$, $(e-1,1)$, corresponding to the matrix $\vec{G}_{q,0}$. When $r>0$, the indices $(e-r,0), \dots, (0,e-r)$ satisfy $i + j = e-r$, while $(e-r+1,e-1), \dots, e-1, e-r+1)$ verify $i+j = 2e -r$. Therefore, $\vec{A} = \vec{G}_{q,r}$, and $\vec{G}$ has the desired shape. \end{demo} 

\begin{proposition}\label{prop:totram-split-nondyadic} Assume that $\Pp$ is a non-dyadic prime ideal and that it splits in $E$. Let $v_1,v_2 \in \Z$ and write $v_1 + v_2 = qe + r$ the Euclidean division of $v_1 + v_2$ by $e$. Then $\L_{v_1,v_2,\lm}$ is $\O_\lm$-isometric to one of the following. \begin{arabiclist} 
\item $ \langle \pi^q \rangle \otimes \langle 1, -1 \rangle \bigperp \langle \pi^{q+1} \rangle \otimes \langle 1, \dots, 1, (-1)^{e-1} \rangle,$ if $r=0;$ \item $\langle \pi^{q+1} \rangle \otimes \langle 1, \dots, 1, (-1)^e \rangle,$ if $r=1;$ \item $\langle \pi^{q+1} \rangle \otimes \underbrace{\langle 1, \dots, 1, (-1)^{e-r+1} \rangle}_{2(e-r+1)} \bigperp \langle \pi^{q+2} \rangle \otimes \underbrace{\langle 1, \dots, 1, (-1)^{r-1} \rangle}_{2(r-1)},$ if $r>1$. \end{arabiclist} \end{proposition} 
\begin{demo} 
Let $e_i = (\Pi^{v_1+i},0)$ and $f_j = (0,\Pi^{v_2+j})$ for $0 \le i,j \le e-1$. Then $\mathcal{B} = \{e_i,f_j\}$ is a basis of $\L = \P^{v_1}\times \P^{v_2}$ over $\O_\lm$. 

1. Assume $r=0$. Set $\L_1 = e_0\O_\lm + f_0\O_\lm$ and $\L_2 = \bigoplus_{i=1}^{e-1}(e_i\O_\lm + f_{e-i}\O_\lm)$. Then $\L = \L_1 \perp \L_2$, with Gram matrices $J_q^{(2)}$ and $\mathrm{diag}(J_{q+1}^{(2)},\dots,J_{q+1}^{(2)})$. Thus $\L_1$ is $\p^q$-modular and $\L_2$ is $\p^{q+1}$-modular. The result then follows from the classification (\ref{eq:nondyadic-modular-classification}) and the computation of the discriminant: 
\[ \disc(\L_1)=\det\bigl(J_q^{(2)}\bigr)=-(e\pi^q)^2 \equiv -\pi^{2q} \pmod{\O_\lm^{\times 2}}, 
\] since $e \in \O_\lm^\times$. Note that $\disc(\L_2) = \det\bigl(J_{q+1}^{(2)}\bigr)^{e-1}$, so this concludes in that case. Assume $r = 1$. We have $\vec{G}_{q,1} = J_{q+1}^{(e)}$, hence $\L$ is $\p^{q+1}$-modular. Its discriminant satisfies $\mathrm{disc}(\L) = (-1)^e \pi^{2e(q+1)} \pmod{\O_\lm^{\times 2}},$ and the conclusion follows from (\ref{eq:nondyadic-modular-classification}). 

2. Assume $r > 0$. Then $\L$ splits into $\L_1 \perp \L_2$, where \[ \L_1 = \bigoplus_{i=0}^{e-r}\Big(e_i\O_\lm + f_{e-r-i}\O_\lm\Big) \quad \text{and} \quad \L_2 = \bigoplus_{i=1}^{r-1}\Big(e_{e-r+i}\O_\lm + f_{e-i}\O_\lm\Big). \] Their Gram matrices are $\mathrm{diag}(J_{q+1}^{(2)},\dots)$ and $\mathrm{diag}(J_{q+2}^{(2)},\dots)$. In particular, $\L_1$ is $\p_\lm^{q+1}$-modular and $\L_2$ is $\p_\lm^{q+2}$-modular. We conclude as in the previous case. 
\end{demo}

\subsubsection{Inert or ramified case}
Depending on the ramification of $\P$ in $E$, we saw that the scaled form $2a x^2$ may appear in the splitting of lattices $\L_{\ida,\lk}$ (\textit{c.f.,} Lemma~\ref{lem:inert-ram-nondyadic}). When descending to $\lm$, one is therefore led to consider the quadratic form $2\tr_{\lk/\lm}(a x^2)$, but the scaling by $a$ makes the computations impractical. The following lemma shows that, in the non-dyadic case, $a$ can be replaced by an element of $\O_\lm^\times \cup \Pi \O_\lm^\times$, which makes the computations feasible. 

\begin{lemma}\label{lem:replace-alpha}
Assuming $\Pp$ non-dyadic, there is a group isomorphism 
$$
\O_\lm^\times / \O_\lm^{\times \, 2} \to \O_\lk^\times / \O_\lk^{\times \, 2}
$$  
\end{lemma}

\begin{demo}
Since $\lk$ is a totally ramified extension of $\lm$, both fields have the same residue field $k$. Fix a set of representatives $S \subset \O_\lm^\times$ for $k^\times$. Then we have
\begin{align*}
\O_\lk^\times = S \cdot U_\lk^{(1)} \quad \text{and} \quad \O_\lm^\times = S \cdot U_\lm^{(1)},
\end{align*}
where $U_\lk^{(1)} = 1 + \P$ and $U_\lm^{(1)} = 1 + \p_\lm$. In particular, $\O_\lk^\times = \O_\lm^\times \cdot U_\lk^{(1)}$. Moreover, under the assumption $2 \notin \P$, Hensel's lemma implies that $U_\lk^{(1)} \subseteq \O_\lk^{\times 2}$.\footnote{If $x \in U_{\lk}^{(1)}$, then $x$ reduces to $1 \in k^{\times 2}$, by definition. Hence $f(X)=X^2-x$ has a simple root modulo $\P$ at $X=1$, since $f'(1)=2 \notin \P$. By Hensel's lemma, $x$ has a square root in $\O_\lk^\times$.} It follows that the natural map $\O_\lm^\times \longrightarrow \O_\lk^\times / \O_\lk^{\times 2}$ is surjective. Finally, its kernel is $\O_\lm^{\times 2}$, again by an application of Hensel's lemma.
\end{demo}

\begin{lemma}\label{lem:Gram-inert-nondyadic}
Assume that $\Pp$ is a non-dyadic prime ideal that is either inert or ramified in $E$, and let $v \in \Z$.   
\begin{arabiclist}
\item If $a \notin \P,$ there exists $\mu_\lm \in \O_\lm^\times$ such that $\L_{v,\lm}$ has Gram matrix 
$$
\vec{G} = \begin{pmatrix}
2\vec{G}_{q,r} & \vec{0} \\ \vec{0} & 2\mu_\lm\vec{G}_{q,r}
\end{pmatrix},
$$
where $qe + r$ is the Euclidean division of $2v$ by $e$.
\item If $a \in \P$, there exists $\mu_\lm \in \O_\lm^\times$ such that $\L_{v,\lm}$ has Gram matrix 
$$
\vec{G} = \begin{pmatrix}
2\vec{G}_{q_1,r_1} & \vec{0} \\ \vec{0} & 2\mu_\lm\vec{G}_{q_2,r_2}
\end{pmatrix},
$$ 
where $q_1 e + r_1$ and $q_2 e + r_2$ are respectively the Euclidean divisions of $2\lceil v/2 \rceil$ and $2\lfloor v/2 \rfloor + 1$ by $e$. 
\end{arabiclist}
\end{lemma}

\begin{demo}
$1.$ By Lemma \ref{lem:inert-ram-nondyadic}, $\L_{v,\lm}$ has a Gram matrix diag$(\vec{A}, \vec{B})$ where $\vec{A}$ (resp. $\vec{B})$ is a Gram matrix of $\P^v$ with respect to the quadratic form $\tr_{\lk/\lm}(2x^2)$ (resp. $\tr_{\lk/\lm}(2a x^2))$. Considering the canonical basis $\{\Pi^v, \dots, \Pi^{v+ e -1}\}$ of $\P^v$, we obtain:
\[
\vec{A} = \left( 2\tr_{\lk/\lm}\left(\Pi^{2v+i+j}\right)\right)_{0 \leq i,j \leq e-1} \ \text{and} \ \vec{B} = \left( 2\tr_{\lk/\lm}(a\Pi^{2v+i+j})\right)_{0 \leq i,j \leq e-1}.
\]
As in the proof of Lemma~\ref{lem:Gram-split}, one checks that $\vec{A} = 2\vec{G}_{q,r}$. By Lemma \ref{lem:replace-alpha}, there exists $\mu_\lm \in \O_\lm^{\times}$ such that  $a^{-1}\mu_\lm \in \O_\lk^{\times 2}$. It follows from a change of variables that $(\P^v, 2\tr_{\lk/\lm}(a x^2)) \simeq_{\O_\lm} (\P^v, 2\mu_\lm\tr_{\lk/\lm}(x^2))$. Therefore, $\vec{B} = 2\mu_\lm \vec{G}_{q,r}$.

$2.$ Again by Lemma \ref{lem:inert-ram-nondyadic}, $\L_{v,\lm}$ has a Gram matrix diag$(\vec{A}, \vec{B})$, but now $\vec{A}$ (resp. $\vec{B})$ is a Gram matrix of $\P^{\lceil v/2 \rceil}$ (resp. $\P^{\lfloor v/2 \rfloor}$) with respect to the quadratic form $\tr_{\lk/\lm}(2x^2)$ (resp. $\tr_{\lk/\lm}(2a x^2))$. Thus, 
\[
\vec{A} = \left( 2\tr_{\lk/\lm}\left(\Pi^{2\lceil v/2 \rceil+i+j}\right)\right)_{0 \leq i,j \leq e-1} \ \text{and} \ \vec{B} = \left( 2\tr_{\lk/\lm}(a\Pi^{2\lfloor v/2 \rfloor+i+j})\right)_{0 \leq i,j \leq e-1}.\] The same computation as above shows that $\vec{A} = 2\vec{G}_{q_1, r_1}$, where $2\lceil v/2 \rceil = q_1e + r_1$. For $\vec{B}$, note that $v_\P(a) = 1$,\footnote{Indeed, $a \in \P$ so $v_\P(a) \geq 1$, but the assumption that $(a)$ has no square factor implies $v_\P(a) \leq 1$.} so we can write $a = \mu\Pi$ for some $\mu \in \O_\lk ^\times$. By Lemma \ref{lem:replace-alpha}, there is a unit $\mu_\lm \in \O_\lm^{\times}$ such that $\vec{B} = (2\mu_\lm \tr_{\lk/\lm}(\Pi^{2\lfloor v/2 \rfloor + i + j + 1})_{0 \leq i,j \leq e-1} = 2\mu_\lm \vec{G}_{q_2,r_2}$, where $2\lfloor v/2 \rfloor  + 1 = q_2 e + r_2$.
\end{demo}

\begin{proposition}\label{prop:totram-inert-nondyadic}
Assume $\Pp$ is a non-dyadic prime ideal, and let $v \in \Z$. There exists a unit $\mu_\lm \in \O_\lm^{\times}$ such that $\L_{v, \lm}$ is $\O_\lm$-isometric to one the following.
\begin{arabiclist}
\item Assume that $a \notin \P$, and let $qe + r$ be the Euclidean division of $2v$ by $e$.
\vspace*{0.2cm}
\begin{alphlist}
\item $\langle \pi^q \rangle \otimes \langle 1, \mu_{\lm} \rangle \bigperp \langle \pi^{q+1} \rangle \otimes \langle 1, \dots, 1, \mu_\lm^{e-1} \rangle$, if $r=0;$ 
\item $\langle \pi^{q+1} \rangle \otimes \underbrace{\langle 1, \dots, 1, \mu_\lm^{e-r+1} \rangle}_{2(e-r+1)} \bigperp \langle \pi^{q+2} \rangle \otimes \underbrace{\langle 1, \dots, 1, \mu_\lm^{r-1} \rangle}_{2(r-1)}$, if $r > 0$.
\end{alphlist}
\item Assume $a \in \P$ and $v$ is even. Let $qe + r$ be the Euclidean division of $v$ by $e$.
\vspace*{0.2cm}
	\begin{alphlist}
	\item $\langle 2e \pi^q \rangle \bigperp \langle \pi^{q+1} \rangle \otimes \langle 1, \dots, 1, (-\mu_\lm)^{e-1}2e\mu_\lm \rangle$, if $r = 0;$

	\item $\langle \pi^{q+1} \rangle \otimes \underbrace{\langle 1, \dots, 1, (-\mu_\lm)^{e-r} 2e \rangle}_{2(e-r) +1} 
	\bigperp \langle \pi^{q+2} \rangle \otimes \underbrace{\langle 1, \dots, 1, (-\mu_\lm)^{r-1} 2e \mu_\lm \rangle}_{2r-1},$ \\ if $r>0$.
	\end{alphlist}
\item  Assume $a \in \P$ and $v$ is odd. Let $qe + r$ be the Euclidean division of $v$ by $e$.
\vspace*{0.2cm}
	\begin{alphlist}
	\item $\langle 2e \mu_\lm \pi^q \rangle \bigperp \langle \pi^{q+1} \rangle \otimes \langle 1, \dots, 1, (-\mu_\lm)^{e-1} 2e \rangle,$ if $r = 0;$
	\item $\langle \pi^{q+1} \rangle \otimes \underbrace{\langle 1, \dots, 1, (-\mu_\lm)^{e-r} 2e\mu_\lm \rangle}_{2(e-r) +1} \bigperp \langle \pi^{q+2} \rangle \otimes \underbrace{\langle 1, \dots, 1, (-\mu_\lm)^{r-1} 2e \rangle}_{2r-1}$, \\ if $r>0.$
	\end{alphlist}
\end{arabiclist}
\end{proposition}

\begin{demo} 
$1.$ By Lemma \ref{lem:replace-alpha}, there exists $\mu_\lm \in \O_\lm^{\times}$ such that $\mu_\lm \equiv a \pmod{\O_\lk^{\times 2}}$. Let us denote $e_i = (\Pi^{v + i}, 0)$, and $f_j = (0, \Pi^{v+j})$ for $0 \leq i,j \leq e-1$, so that Lemma \ref{lem:Gram-inert-nondyadic} \textit{1.} is the Gram matrix of $\{ e_i, f_j \}$ with respect to $2\tr_{\lk/\lm}(x^2) \perp 2\mu_\lm\tr_{\lk/\lm}(x^2)$.

$a.$ If $r=0$, this matrix is diag($2e\pi^q, 2J_{q+1}^{(e-1)}, 2\mu_\lm e\pi^q, 2 \mu_\lm J_{q+1}^{(e-1)})$, and thus, $\L_0 = e_0 \O_\lm + f_0 \O_\lm$ is a $\p_\lm^q$-modular sublattice of rank two. Moreover, its discriminant is $4e^2\mu_\lm \pi^{2q} \equiv \mu_\lm \pi^{2q} \pmod{\O_\lm^{\times 2}}$, so $\L_0 \simeq_{\O_\lm} \langle \pi^q \rangle \otimes \langle 1, \mu_\lm \rangle$. The sublattice spanned by the other vectors is orthogonal to $\L_0$ and it is a $\p_\lm^{q+1}$-modular lattice. Its discriminant is 
\begin{align*}
\det(2J_{q+1}^{(e-1)}) \cdot \det(2 \mu_\lm J_{q+1}^{(e-1)}) &= (2e)^{2(e-1)} \mu_{\lm}^{e-1} \pi^{2(e-1)(q+1)} \\ &\equiv \mu_\lm^{e-1} \pi^{2(e-1)(q+1)} \pmod{\O_\lm^{\times 2}}.
\end{align*}
This gives the result in this case. $b.$ is treated similarly.
\\

$2.$ We fix $\mu_\lm$ such that $\mu_\lm \equiv \Pi^{-1}a \pmod{\O_\lk^{\times 2}}$. As $v$ is even, a Gram matrix of $\L_{v,\lm}$ is, according to Lemma \ref{lem:Gram-inert-nondyadic} \textit{2.}, diag($2\vec{G}_{q,r}, 2\mu_\lm\vec{G}_{q',r'})$, where $qe + r$ (resp. $q'e + r'$) is the Euclidean division of $v = 2 \lceil \frac{v}{2} \rceil$ (resp. $v+1 = 2\lfloor \frac{v}{2} \rfloor +1$) by $e$.

$a.$ If $r=0$, then $q'=q$ and $r' = r+1$, so that the Gram matrix of $\L_{v,\lm}$ becomes diag$(2e\pi^q, 2J_{q+1}^{(e-1)}, 2\mu_\lm J_{q+1}^{(e)})$. Thus we extract the rank-one sublattice $\langle 2e\pi^q \rangle$ and the $\p_\lm^{q+1}$-modular part, whose discriminant is 
\[
(-1)^{\lfloor \frac{e-1}{2} \rfloor + \lfloor \frac{e}{2} \rfloor} (2e)^{2e -1} \mu_\lm^{e} \pi^{(2e-1)(q+1)} \equiv (-\mu_\lm)^{e-1}2e\mu_\lm \pi^{(2e-1)(q+1)} \pmod{\O_\lm^{\times 2} },
\] 
since $\lfloor \frac{e-1}{2} \rfloor + \lfloor \frac{e}{2} \rfloor = e-1$. The other cases are treated similarly.


$b.$ If $r = 1$ and $e \geq 3$, then $q' = q$ and $r' = r+1$. In that case, the second summand is empty, \textit{i.e.,} $\L_{v,\lm}$ is $\p_\lm^{q+1}$-modular. Note that $(r,e) = (1,2)$ is incompatible with the assumption that $v$ is even. If $e \geq 3$ and $1 < r < e-1$, then $q'=q$ and $r' = r+1$. Finally if $e \geq 3$ and $r = e-1$, then $q'=q+1$ and $r'=0$.
\\ 

$3.$ We choose $\mu_\lm$ as in $2.$ Since $v$ is odd, we have $2\lceil \frac{v}{2} \rceil = v +1$ and $2\lfloor \frac{v}{2} \rfloor + 1 = v$. Hence, the situation is similar to $2.$, but the roles of $v$ and $v+1$ are swapped.
\end{demo}

\subsection{Over the unramified extension}\label{sec:unr}

Recall that $\lm/\lf$ denotes an unramified extension of local fields, with degree $f$. The following proposition, adapted from \cite[Lemma 1.4]{mantillagenus}, shows how, from a splitting of $\L_{\ida, \lm}$ over $\O_\lm$, one can derive a splitting over $\O_\lf$, \textit{i.e.,} a splitting of the $\O_\lf$-lattice
$$
\L_{\ida, \lf} = \Big(\ida \otimes \O_\lk, \tr_{\lk/\lf} \circ \tr_{E_\P/\lk}(x\overline{y})\Big).
$$

\begin{lemma}\label{prop:interm1}
Let $(\L,q)$ be a quadratic $\O_\lm$-lattice. Assume that $\p$ is non-dyadic and that there exist $r \in \Z_{\geq 0}$ and $\alpha_0, \dots, \alpha_{r-1} \in \O_\lm$ such that $(\L,q) \simeq_{\O_\lm} \langle \alpha_0, \dots, \alpha_{r-1} \rangle.$ Over $\O_\lf$, we have
\[
(\L, \tr_{\lm/\lf} \circ q) \simeq_{\O_\lf} \bigperp_{i=0}^{r-1} \left(\O_\lm, q_{\lm/\lf}^{\alpha_i}\right),
\]
where $q_{\lm/\lf}^{\alpha_i}$ is the scaled traced form defined by $q_{\lm/\lf}^{\alpha_i}(x) = \tr_{\lm/\lf}(\alpha_i x^2).$
\end{lemma}

\begin{demo}
By assumption, there is a basis  $\{ \vec{v}_0, \dots, \vec{v}_{r-1} \}$ of $\L$ over $\O_\lm$ such that $(\L,q)$ is the orthogonal sum of $(V_i,q) := (\vec{v}_i \O_\lm, q)$. For $0 \leq i \leq r-1$ and each element $\vec{v} = \vec{v}_i x \in V_i$, note that $\tr_{\lm/\lf}(q(\vec{v})) = \tr_{\lm/\lf}(x^2 q(\vec{v}_i)) = \tr_{\lm/\lf}(\alpha_i x^2).$  Thus, $(V_i, \tr_{\lm/\lf} \circ q) \simeq_{\O_{\lf}} (\O_\lm, q_{\lm/\lf}^{\alpha_i})$. The orthogonality of the \(V_i\)'s is preserved, since applying
\(\tr_{\lm/\lf}\) to the vanishing cross-terms again gives zero.
\end{demo}

\begin{lemma}\label{prop:decomp_unramified}
Assume that $a\in \O_\lm ^\times$ is a unit. Then, $(\O_{\hf}, {q_{\hf/\hk}^{a}})$ is a unimodular $\O_\lf$-lattice of rank $f$, and we have the isometry
\begin{align*}
(\O_{\hf}, {q_{\hf/\hk}^{a}}) \simeq_{\O_{\hk}} \langle 1, \dots, 1, \norm_{\hf/\hk}(a) \cdot d_{\lm/\lf} \rangle,
\end{align*}
where $d_{\lm/\lf}$ denotes the discriminant of the local extension $\lm/\lf$.
\end{lemma}
\begin{demo}
We have the equality of discriminants 
$$
\mathrm{disc}(\O_{\hf}, {q_{\hf/\hk}^{a}}) = \norm_{\hf/\hk}(a) \cdot d_{\hf/\hk}.
$$ 
Since $\lm/\lf$ is unramified and $a$ is unit, we conclude that the discriminant of $(\O_{\hf}, {q_{\hf/\hk}^{a}})$ is a unit too, hence it is a unimodular lattice. The isometry is then a direct consequence of (\ref{eq:nondyadic-modular-classification}).
\end{demo}

\begin{remark}
In the lemma above, one can replace $d_{\lm/\lf}$ by $u_\p^{f-1}$, where $u_\p$ is any element of $\O_\lf^{\times}\setminus \O_\lf ^{\times 2}$ (\textit{c.f.}, \cite[Proposition 1.7]{mantillagenus}).
\end{remark}

\subsection{Gluing the decompositions}\label{sec:gluing-non-dyadic}

We carry on the notation introduced in Section~\ref{sec:totram}. Combining Lemma~\ref{prop:interm1} and Lemma~\ref{prop:decomp_unramified} with the splittings of $\L_{\ida,\lm}$ obtained in Section~\ref{sec:totram}, we derive a splitting over $\O_\lf$.

Since $\lm/\lf$ is unramified, there exists a unit $\nu_\lm \in \O_\lm^\times$ such that $\nu_\lm \cdot \pi$
is a uniformizer of $\lf$. We denote this uniformizer by $\pi_\lf$. We also define $\nu_\lf := \norm_{\lm/\lf}(\nu_\lm)$ and $\mu_\lf := \norm_{\lm/\lf}(\mu_\lm)$.
\subsubsection{Split case}

\begin{proposition}\label{glue-non-dyadic-split}
Assume that $\Pp$ is non-dyadic and that it splits in $E$. Let $v_1,v_2 \in \Z$ and write $v_1+v_2 = qe+r$ for the Euclidean division of $v_1+v_2$ by $e$. Then, $\L_{v_1,v_2,\lf}$ is
$\O_\lf$-isometric to one of the following.
\begin{arabiclist}
    \item  $\langle \pi_\lf^{q}\rangle
    \otimes
    \underbrace{\langle 1,\dots,1,(-1)^f\rangle}_{2f}
    \bigperp
    \langle \pi_\lf^{q+1}\rangle
    \otimes
    \underbrace{\langle 1,\dots,1,(-1)^{f(e-1)}\rangle}_{2f(e-1)}$, if $r=0;$

    \item $\langle \pi_\lf^{q+1}\rangle
    \otimes
    \underbrace{\langle 1,\dots,1,(-1)^{f(e-r+1)}\rangle}_{2f(e-r+1)}
    \bigperp
    \langle \pi_\lf^{q+2}\rangle
    \otimes
    \underbrace{\langle 1,\dots,1,(-1)^{f(r-1)}\rangle}_{2f(r-1)}$, if~$r>~0$.
\end{arabiclist}
\end{proposition}

\begin{demo}
We combine Proposition~\ref{prop:totram-split-nondyadic} to obtain a Jordan decomposition over
$\O_\lm$ with Lemma~\ref{prop:interm1} and Lemma~\ref{prop:decomp_unramified}
on each rank-one summand. Since all resulting modular blocks have even rank, the
unit factors arising from $d_{\lm/\lf}$ and $\nu_\lf$ are squares in $\O_\lf^\times$,
and the claim follows from the non-dyadic classification.
\end{demo}

\subsubsection{Inert or ramified case}

\begin{proposition}\label{glue-non-dyadic-inert-ramified}
Assume that $\Pp$ is non-dyadic and that it is either inert or ramified in $E$. Let $v \in \Z$ and $\varepsilon \in \{0, 1 \}$ such that $v \equiv \varepsilon \pmod{2}$. Then, $\L_{v,\lf}$ is
$\O_\lf$-isometric to one of the following.
\begin{arabiclist}
    \item Assume that $a \notin \Pp$, and let $qe + r$ be the Euclidean
    division of $2v$ by $e$.
    \begin{alphlist}
        \item $
        \langle \pi_\lf^q \rangle \otimes
        \underbrace{\langle 1, \dots, 1, \mu_\lf \rangle}_{2f}
        \bigperp
        \langle \pi_\lf^{q+1} \rangle \otimes
        \underbrace{\langle 1, \dots, 1, \mu_\lf^{e-1} \rangle}_{2f(e-1)}$,
        if $r=0;$


        \item $
        \langle \pi_\lf^{q+1} \rangle \otimes
        \underbrace{\langle 1, \dots, 1, \mu_\lf^{e-r+1} \rangle}_{2f(e-r+1)}
        \bigperp
        \langle \pi_\lf^{q+2} \rangle \otimes
        \underbrace{\langle 1, \dots, 1, \mu_\lf^{r-1} \rangle}_{2f(r-1)}$,
        if $r>0$.\footnote{Note that in the case $r=1$, the second summand vanishes and $\L_{v,\lm}$ is $\p^{q+1}$-modular.}
    \end{alphlist}

    \item Assume that $a \in \Pp$, and let $qe + r$
    be the Euclidean division of $v$ by $e$.
    \begin{alphlist}
        \item $\begin{aligned}[t]
        \langle \pi_\lf^q \rangle \otimes
        &\underbrace{\langle 1,\dots,1,
        \nu_\lf^{-q}d_{\lm/\lf}(2e)^f\mu_\lf^\varepsilon\rangle}_{f} \\
        \bigperp\;
        &\langle \pi_\lf^{q+1} \rangle \otimes
        \underbrace{\langle 1,\dots,1,
        (-1)^{(e-1)f}\nu_\lf^{-(q+1)}d_{\lm/\lf}(2e)^f\mu_\lf^{e-\varepsilon}
        \rangle}_{f(2e-1)},
        \ \text{if } r=0;
        \end{aligned}$

        \item $\begin{aligned}[t]
        \langle \pi_\lf^{q+1} \rangle \otimes
        &\underbrace{\langle 1,\dots,1,
        (-1)^{(e-r)f}\nu_\lf^{-(q+1)}d_{\lm/\lf}(2e)^f\mu_\lf^{e-r+\varepsilon}
        \rangle}_{f(2(e-r)+1)} \\
        \bigperp\;
        &\langle \pi_\lf^{q+2} \rangle \otimes
        \underbrace{\langle 1,\dots,1,
        (-1)^{(r-1)f}\nu_\lf^{-(q+2)}d_{\lm/\lf}(2e)^f\mu_\lf^{r-\varepsilon}
        \rangle}_{f(2r-1)}, \
        \text{if } r>0.
        \end{aligned}$
    \end{alphlist}
\end{arabiclist}
\end{proposition}

\begin{demo}
1. Assume first that $a \notin \Pp$. By
Proposition~\ref{prop:totram-inert-nondyadic}, the lattice $\L_{v,\lm}$ is an
orthogonal sum of blocks of the form $\langle \pi^s \rangle \otimes \langle 1,\dots,1,\mu_\lm^m\rangle,$
with ranks $2$, $2(e-1)$, $2(e-r+1)$ and $2(r-1)$ according to the value of
$r$. After descending to $\O_\lf$, each such block becomes $\langle \pi_\lf^s \rangle \otimes \langle 1,\dots,1,\mu_\lf^m\rangle,$
because all these ranks are even, so the unit factors coming from
$d_{\lm/\lf}$ and $\nu_\lf$ are squares in $\O_\lf^\times$. This gives
\textit{1.a} and \textit{1.b}.
\\ \ \\
2. Assume now that $a \in \Pp$. By
Proposition~\ref{prop:totram-inert-nondyadic}, the lattice $\L_{v,\lm}$ is an
orthogonal sum of one or two modular blocks, whose last coefficients are either $2e\mu_\lm^\varepsilon$ or $(-\mu_\lm)^{e-1}2e\mu_\lm^{1-\varepsilon}$
when $r=0$, and $(-\mu_\lm)^{e-r}2e\mu_\lm^\varepsilon$ or
$(-\mu_\lm)^{r-1}2e\mu_\lm^{1-\varepsilon}$ when $r>0$. Applying
Lemma~\ref{prop:interm1} and Lemma~\ref{prop:decomp_unramified} to each
rank-one summand, and by definition of$\mu_\lf$ and $\nu_\lf$,
we obtain exactly the coefficients appearing in \textit{2.a} and \textit{2.b}.
Moreover, the signs come from $\norm_{\lm/\lf}\big((-\mu_\lm)^t\big)=(-1)^{tf}\mu_\lf^t.$ This proves the proposition.
\end{demo}

\subsubsection{Jordan decomposition in the non-dyadic case}

\begin{theorem}\label{thm:jordan-decomp-nondyadic}
Assume $\p$ is non-dyadic. Let $\ida \subset E$ be a non-zero fractional ideal, and for every prime ideal $\Pp \mid \p$, denote by $\ida_{\Pp} := \ida \otimes_{\O_K} \O_{K_\Pp}$ the $\Pp$-part of $\ida$.
Then, locally at $\p$, the ideal lattice $\L_{\ida, F}$ is isometric to
\[
\bigperp_{\Pp \mid \p}
\L_{\ida_{\Pp},F_\p},
\]
where, each factor $\L_{\ida_{\Pp},F_\p}$ is explicitly described as follows.
\begin{arabiclist}
    \item If $\Pp$ splits in $E$ and if
    $\ida_{\Pp}=\Qq_1^{v_1}\Qq_2^{v_2}$, then
    $\L_{\ida_{\Pp},F_\p}=\L_{v_1,v_2,F_\p}$ and its splitting is given by Proposition~\ref{glue-non-dyadic-split}.

    \item If $\Pp$ is either inert or ramified in $E$ and if $\ida_{\Pp}=\Qq^v$, then
    $\L_{\ida_{\Pp},F_\p}=\L_{v,F_\p}$ and its splitting is
given by Proposition~\ref{glue-non-dyadic-inert-ramified}.
\end{arabiclist}

Additionally, a Jordan decomposition of $\L_{\ida,F}$ at $\p$ is obtained by
grouping together, for each integer $s$, the $\p^s$-modular blocks arising
in the decompositions of the lattices $\L_{\ida_\Pp,F_\p}$, as $\Pp$ ranges over
the prime ideals above $\p$.
\end{theorem}

\begin{demo}
Recall the orthogonal splitting from Lemma \ref{lemma:orthoplit}:
\[
\bigl(\ida \otimes_{\O_F} \O_{F_\p}, h_{E_\p/F_\p}\bigr)
\simeq_{\O_{F_\p}}
\bigperp_{\Pp \mid \p}
\bigl(\ida_{\Pp}, h_{E_\Pp/F_\p}\bigr).
\]
For each $\Pp \mid \p$, the $\Pp$-part $\ida_{\Pp}$ of $\ida$ determines the
corresponding factor. If $\Pp$ splits in $E$, then
$\ida_{\Pp}=\Qq_1^{v_1}\Qq_2^{\,v_2}$ and Proposition~\ref{glue-non-dyadic-split}
gives the Jordan decomposition of the associated lattice $\L_{v_1,v_2,F_\p}$.
If $\Pp$ is inert or ramified in $E$, then $\ida_{\Pp}=\Qq^v$ and
Proposition~\ref{glue-non-dyadic-inert-ramified} gives the Jordan decomposition
of the associated lattice $\L_{v,F_\p}$. Taking the orthogonal sum over all
$\Pp \mid \p$ yields the claimed decomposition.
\end{demo}

In general, Theorem~\ref{thm:jordan-decomp-nondyadic} does \emph{not} yield a
simple ``if and only if'' criterion for deciding whether two ideal lattices are locally isometric at
$\p$, based on their prime ideal factorization. The reason is that the orthogonal summands attached to different primes
$\Pp \mid \p$ may contribute Jordan blocks with the same scale, and
these components must then be merged before applying the
classification (\ref{eq:nondyadic-modular-classification}). In particular, the total $\p^s$-modular part of
$\ida \otimes_{\O_F} \O_{F_\p}$ may arise from several distinct blocks, so that the relevant invariants can't be read off prime by
prime.

In other words, the obstruction comes from the possible overlap between the sets
of scales appearing in the various factors $\L_{\ida_\Pp,F_\p}$. To simplify the
combinatorics, we now restrict to the case where the contributions of the
different primes $\Pp \mid \p$ occur at disjoint scales. This motivates the
following definition and notation.

\begin{definition}\label{defi:scale-set}
Let $\ida \subset E$ be a non-zero fractional ideal, and let $\p$ be a prime ideal of $\O_F$. For each prime divisor $\Pp \mid \p$ in $\O_K$, we denote by $\mathrm{Sc}_\Pp(\ida)$ the set of integers $s$ such that a $\p^s$-modular block appears in the decomposition of $\L_{\ida_\Pp,F_\p}$.
\end{definition}

\begin{remark}\label{rmk:u_p(a)}
Thanks to Proposition \ref{glue-non-dyadic-split} and Proposition \ref{glue-non-dyadic-inert-ramified}, one can easily describe $\mathrm{Sc}_\Pp(\ida)$ in terms of the valuations of $\ida$. 
Indeed, define the integer $u_\Pp(\ida)$ by
\[
u_\Pp(\ida):=
\begin{cases}
v_{\Qq_1}(\ida)+v_{\Qq_2}(\ida),
& \text{if } \Pp\O_E=\Qq_1\Qq_2 \text{ is split in }E;\\[1ex]
2\,v_\Qq(\ida)
& \text{if } \Pp\O_E=\Qq \text{ is inert in }E;\\[1ex]
v_\Qq(\ida),
& \text{if } \Pp\O_E=\Qq^2 \text{ is ramified in }E.
\end{cases}
\]
Moreover, write $u_\Pp(\ida)=q_\Pp e_\Pp +r_\Pp$
for the Euclidean division of $u_\Pp(\ida)$ by the ramification index $e_\Pp$ at $\P \mid \p$. Then, the set of scales
contributed by $\Pp$ is
\[
\mathrm{Sc}_\Pp(\ida) =
\begin{cases}
\{q_\Pp,q_\Pp+1\},
& \text{if } r_\Pp=0;\\[1ex]
\{q_\Pp+1\},
& \text{if } r_\Pp=1 \text{ and } \Pp \text{ is split or inert in }E;\\[1ex]
\{q_\Pp+1,q_\Pp+2\},
& \text{otherwise.}
\end{cases}
\]
\end{remark}

Thus, for each prime $\Pp\mid\p$, the possible scales contributed by
$\L_{\ida_\Pp,F_\p}$ are completely encoded by the finite set
$\mathrm{Sc}_\Pp(\ida)$. The situation in which these sets are pairwise
disjoint is precisely the one where no merging of Jordan blocks can occur
between distinct primes above $\p$. This motivates the following definition.

\begin{definition}\label{defi:p-separated}
We say that $\ida$ is \emph{$\p$-separated} if, for every  $\Pp' \neq \Pp$ above~$\p$,
\[
\mathrm{Sc}_\Pp(\ida)\cap \mathrm{Sc}_{\Pp'}(\ida)=\varnothing.
\]
\end{definition}

Before turning this into a usable local criterion, let us first illustrate the
condition with a few examples of \(\p\)-separated ideals.

\begin{example}
The condition of being $\p$-separated is satisfied when the ideal
$\ida$ is supported, above $\p$, at only one prime ideal, provided that this
nontrivial contribution is far enough from the trivial ones. 

For instance,
assume that there are exactly two prime ideals $\Pp_1,\Pp_2$ of $\O_K$ above
$\p$, both with ramification index \(e_{\Pp_i}=1\) over \(\p\). Suppose that
\[
u_{\Pp_1}(\ida)=0
\qquad\text{and}\qquad
u_{\Pp_2}(\ida)=N,
\]
for some integer \(N\geq 2\). Then, by Remark~\ref{rmk:u_p(a)}, one has $\mathrm{Sc}_{\Pp_1}(\ida)=\{0,1\}$ and $\mathrm{Sc}_{\Pp_2}(\ida)=\{N,N+1\}.$
These two sets are disjoint, and hence \(\ida\) is \(\p\)-separated.
\end{example}

\begin{example}
We now give a less degenerate example. Let $F=\Q, K=\Q(\sqrt{5})$ and $E=K(i)$ and take \(p=11\). Since \(5\) is a square modulo \(11\), the prime \(11\)
splits in \(K\): $11\O_K=\Pp_1\Pp_2.$
Moreover, the residue fields at \(\Pp_1\) and \(\Pp_2\) are both isomorphic to
\(\mathbb{F}_{11}\). Since \(-1\) is not a square modulo \(11\), the primes
\(\Pp_1\) and \(\Pp_2\) are inert in \(E/K\). We write $\Pp_i\O_E=\Qq_i,$ for $i \in \{1,2\}$.
Let
\[
\ida=\Qq_1\Qq_2^3.
\]
Then $u_{\Pp_1}(\ida)=2v_{\Qq_1}(\ida)=2$ and $u_{\Pp_2}(\ida)=2v_{\Qq_2}(\ida)=6.$
Since \(e_{\Pp_1}=e_{\Pp_2}=1\), the Euclidean divisions by \(e_{\Pp_i}\)
have remainder \(0\). By Remark~\ref{rmk:u_p(a)}, we get
\[
\mathrm{Sc}_{\Pp_1}(\ida)=\{2,3\},
\qquad
\mathrm{Sc}_{\Pp_2}(\ida)=\{6,7\}.
\]
These two sets are disjoint. Thus \(\ida\) is \(11\)-separated.
\end{example}

Assume now that both $\ida$ and $\idb$ are $\p$-separated. Then the overlap
phenomenon described above disappears: By definition, for each of the two ideals,
every scale $\p^s$ occurs in at most one factor $\L_{\ida_\Pp,F_\p}$ (resp.
$\L_{\idb_\Pp,F_\p}$). Hence, after applying Theorem~\ref{thm:jordan-decomp-nondyadic},
deciding whether $\L_{\ida,F}$ and $\L_{\idb,F}$ are locally isometric at $\p$
amounts to comparing the Jordan decompositions of the factors
$\L_{\ida_\Pp,F_\p}$ and $\L_{\idb_{\Pp'},F_\p}$ prime by prime (where $\Pp, \Pp' \mid \p)$, up to a suitable
reordering of the primes above $\p$. 

Moreover, Proposition~\ref{glue-non-dyadic-inert-ramified}
shows that the ramified primes are the only ones contributing modular blocks whose
ranks are of the form $f \cdot (2m+1)$, whereas Propositions~\ref{glue-non-dyadic-split}
and~\ref{glue-non-dyadic-inert-ramified} show that the split and inert primes
contribute only blocks of rank $f \cdot 2m$. In particular, the ramified factors
must necessarily be paired with ramified factors. 

On the other hand, in the split and
inert cases, the possible ranks are the same. The distinction therefore lies in
the discriminant classes of the corresponding modular blocks, namely
$(-1)^f$ in the split case and
$\mu_F$ in the inert case. The following lemma determines precisely when these two coincide modulo $\O_\lf^{\times 2}$. 

\begin{lemma}\label{lem:muF-minusone}
Assume that $\Pp \mid \p$ is a non-dyadic prime ideal which is inert in $E$, and let $\mu_\lm \in \O_\lm^\times$ be as in Lemma~\ref{lem:Gram-inert-nondyadic}, namely such that $a^{-1}\mu_\lm \in \O_\lk^{\times 2}.$ 
Recall that $\mu_\lf$ is defined by $\mu_\lf = \norm_{\lm/\lf}(\mu_\lm)$. Then we have the equivalence between
\begin{arabiclist}
\item $\mu_\lf \equiv (-1)^f \pmod{\O_\lf^{\times 2}};$
\item $f \text{ is even}.$
\end{arabiclist}
\end{lemma}

\begin{demo}
Let $k_\lm$ and $k_\lf$ denote the residual fields of $\O_\lm$ and $\O_\lf$, respectively. Since $\P$ is non-dyadic, Hensel's lemma implies that both
$\O_\lm^\times/\O_\lm^{\times 2} \simeq k_\lm^{\times}/k_\lm^{\times 2}$ and $\O_\lf^\times/\O_\lf^{\times 2} \simeq k_\lf^{\times}/k_\lf^{\times 2}$ have order two. Moreover, the norm on $\lm/\lf$ induces 
\[
k_\lm^{\times}/k_\lm^{\times 2} \longrightarrow k_\lf^{\times}/k_\lf^{\times 2},
\]
which is therefore trivial or an isomorphism. Since $[k_\lm:k_\lf]=f$, this map is given on the quotient of order $2$ by $x \mapsto x^f$. Therefore, $[\mu_\lf]$ identifies with $[\mu_\lm]^f
\ \text{in }\O_\lf^\times/\O_\lf^{\times 2}.$
In particular, if $f$ is even, we obtain the claim
\[
\mu_\lf \equiv 1 \pmod{\O_F^{\times 2}}.
\]

Assume now that $f$ is odd. Then $[\mu_\lf]$ and $[\mu_\lm]$ coincide in $\O_\lf^\times/\O_\lf^{\times 2}$. Suppose for contradiction
that we have
\[
\mu_\lf \equiv (-1)^f = -1 \pmod{\O_\lf^{\times 2}}.
\]
This implies $\mu_\lm \equiv -1 \pmod{\O_\lm^{\times 2}}.$ But then by definition of $\mu_\lm$, we obtain
\[
\mu_\lm \equiv a \pmod{\O_\lk^{\times 2}}, \quad \text{hence} \ a \equiv -1 \pmod{\O_\lk^{\times 2}}.
\]
It follows that $-a$ is a square in $\lk$, so the quadratic extension
$K(\sqrt{-a})/K$ splits at $\P$, contradicting the assumption that $\P$ is
inert in $E$. This proves the result.
\end{demo}

\begin{corollary}[Non-dyadic criterion for $\p$-separated ideals]\label{cor:local-criterion-nondyadic}
Assume that $\p$ is non-dyadic, that $K/F$ is a Galois extension, and that the fractional ideals $\ida$ and $\idb$ are $\p$-separated (\textit{c.f.,} Definition \ref{defi:p-separated}).  Let $f$ denote the residual degree at $\p$,\footnote{Since $K/F$ is Galois, $f = [\O_K/\Pp : \O_F/\p]$ does not depend on the choice of $\Pp \mid \p$.} and recall the notation $u_{\Pp}(\ida)$ (\textit{c.f.,} Remark \ref{rmk:u_p(a)}).

Then, $\L_{\ida, F} \text{ and } \L_{\idb, F} \text{ are locally isometric at } \p$ if and only if:

\begin{arabiclist}
    \item If $f$ is odd, then the multisets
    \[
    \{u_{\Pp}(\ida)\}_{\Pp \text{ split in } E} \quad ; \quad  \{u_{\Pp}(\ida)\}_{\Pp \text{ inert in } E} \quad ; \quad  \{u_{\Pp}(\ida)\}_{\Pp \text{ ramified in } E}   
    \]
    coincide respectively with the corresponding multisets for $\idb$.

    \item If $f$ is even, then the multisets
    \[
    \{u_{\Pp}(\ida)\}_{\Pp \text{ split or inert in } E} \quad ; \quad  \{u_{\Pp}(\ida)\}_{\Pp \text{ ramified in } E}   
    \]
    coincide respectively with the corresponding multisets for $\idb$.
\end{arabiclist}
\end{corollary}

\begin{proof}
By Theorem~\ref{thm:jordan-decomp-nondyadic}, one has orthogonal decompositions
\[
\L_{\ida, F_\p}
\simeq_{\O_{F_\p}}
\bigperp_{\Pp\mid\p}\L_{\ida_{\Pp},F_\p}
\qquad\text{and}\qquad
\L_{\idb, F_\p}
\simeq_{\O_{F_\p}}
\bigperp_{\Pp\mid\p}\L_{\idb_{\Pp},F_\p}.
\]
Since $\ida$ and $\idb$ are $\p$-separated, no merging occurs between distinct
local factors in the orthogonal decompositions given by
Theorem~\ref{thm:jordan-decomp-nondyadic}. Hence the two lattices are locally
isometric at $\p$ if and only if, after a suitable reordering of the primes
above $\p$, the factors $\L_{\ida_\Pp,F_\p}$ and $\L_{\idb_{\Pp'},F_\p}$ can be
matched pairwise.

At this point, the assumption that $K/F$ is Galois becomes essential. Indeed,
it implies that all primes $\Pp \mid \p$ have the same ramification index $e$
and the same residual degree $f$, and that the local extensions $K_\Pp/F_\p$
are all $F_\p$-isomorphic. Therefore, the explicit formulas of
Propositions~\ref{glue-non-dyadic-split} and
\ref{glue-non-dyadic-inert-ramified} show that the $\O_{F_\p}$-isometry class
of a factor $\L_{\ida_\Pp,F_\p}$ depends only on the splitting type of $\Pp$ in
$E$ and on the integer $u_\Pp(\ida)$.

By Proposition~\ref{glue-non-dyadic-inert-ramified}, the ramified primes are
exactly those contributing modular blocks whose ranks are odd multiples of $f$,
whereas split and inert primes contribute only even multiples of $f$. Hence
ramified factors must be paired with ramified factors. For split and inert
primes, the possible ranks coincide, and the distinction is made by the
discriminant classes $(-1)^f$ and $\mu_F$. By Lemma~\ref{lem:muF-minusone},
these classes coincide if and only if $f$ is even. This concludes the proof.
\end{proof}


\section{The dyadic case}\label{sec:dyadic}

We now turn to the case of dyadic prime ideals, that is, we assume throughout
this section that $2 \in \p$. In contrast with the non-dyadic case, the local
classification of quadratic lattices over dyadic fields involves additional
invariants such as the norm group and the weight (\textit{c.f.,} Section \ref{subsec:dyadic-prelim}), and explicit
splittings are therefore more delicate to obtain. For this reason, and in some
situations, we will have to impose additional assumptions.

An obstacle already appears at the level of square classes of units.
In the non-dyadic case, Lemma~\ref{lem:replace-alpha} allowed us to replace
certain units of $\O_\lk^\times$ by units of $\O_\lm^\times$, which was a key step
in the descent from $\lk$ to $\lm$. This argument breaks down in the dyadic setting,
because the natural map 
\[
\O_\lm^\times/\O_\lm^{\times 2}\longrightarrow \O_\lk^\times/\O_\lk^{\times 2}
\]
is no longer an isomorphism; see the remark below. This is the reason why the
ramified dyadic primes containing $a$ will be excluded from the explicit
descent from $\lk$ to $\lm$ carried out below. In practice, this restriction is often harmless: In many natural CM extensions
of the form $E=K(\sqrt{-a})$, the element $a$ is a global unit, so
that no prime ideal can divide $a$.

\begin{remark}\label{rem:bla}
For a finite extension $\mathbb{K}/\Q_2$ of degree $d$, the quotient $\O_\mathbb{K}^\times / \O_\mathbb{K}^{\times 2}$ is isomorphic to $\mathbb{F}_2^{d+1}$. This fact can be derived from \cite[Chapter~II, Proposition~5.7]{neukirch2013algebraic}. In particular, in our setting, $\O_\lm^\times / \O_\lm^{\times 2} \not\simeq \O_\lk^\times / \O_\lk^{\times 2}.$ Nevertheless, note that there is still a natural inclusion, and that the index is:
\[
\bigl[ \O_\lk^\times / \O_\lk^{\times 2} : \O_\lm^\times / \O_\lm^{\times 2} \bigr] = 2^{(e-1)[\lm \, : \, \Q_2]}.
\]
\end{remark}

\subsection{Over the CM extension} We begin, as in Section \ref{sec:cm}, by studying the local lattice $\L_{\ida,\lk}$
over $\lk = K_\Pp$. The split case requires no new argument:
Lemma~\ref{lem:split} remains valid for dyadic primes as well. Thus the only new phenomenon occurs when $\Pp$ is inert or ramified in $E$, and the purpose of
this subsection is to complete the analogue of Lemma~\ref{lem:inert-ram-nondyadic}
under the assumption $2 \in \Pp$.

\begin{lemma}\label{lem:inert-ram-dyadic} Suppose that $2 \in \Pp$ and let $v \in \Z$. Assume that $\Pp\O_E = \Qq$ or $\Qq^2$ is either inert or ramified in $E$. Let $\varepsilon := v_\P(2)$ and $0 \leq s \leq \varepsilon$ as in Lemma \ref{lem}.
\begin{arabiclist}
	\item If $a \in \P$, then $\P$ is ramified, and 
	$$
	\L_{v, \lk} \simeq_{\O_{\lk}} (\P^{\lceil v/2 \rceil}, 2x^2) \bigperp (\P^{\lfloor v/2 \rfloor}, 2{a}x^2).
	$$
	\item If $a \notin \P$, let $\rho \in \O_{\lk}^\times$ such that $\idd(1-4\rho) = 4\O_{\lk}$ (\textit{c.f.,} Section \ref{subsec:dyadic-prelim}).
	\begin{alphlist}
	\item Suppose $s=\varepsilon$, then $\P$ is inert and
	$$
	\L_{v, \lk} \simeq_{\O_{\lk}} \langle \Pi^{2v} \rangle \otimes A(2,2\rho).
	$$
	\item Suppose $s < \varepsilon$, then $\P$ is ramified and 
	$$
	\L_{v, \lk} \simeq_{\O_{\lk}} \langle \Pi^{\varepsilon-s + v} \rangle \otimes A(2,2\rho).
	$$
	\end{alphlist}
\end{arabiclist} 
\end{lemma}

\begin{demo}
Denote by $\mathcal{E}$ the completion of $E$ at $\Qq$.

$1.$ If $a \in \P$, then by Lemma \ref{lem} we have $\O_{\mathcal{E}} = \O_\lk[\sqrt{-a}]$ and $\Pp$ is ramified, thus the proof is identical to the one for $2.$ in Lemma \ref{lem:inert-ram-nondyadic}. 

$2.$ Suppose now that $a \notin \P$. It follows from Lemma \ref{lem} that $\{1, \theta\}$ is a basis of $\O_{\mathcal{E}}$ over $\O_{\lk}$, where $\theta = (\beta_{\P} + \sqrt{-a})/\Pi^s$ and $\beta_\P$ is a unit satisfying $\beta_{\P}^2 \equiv - a \pmod{\P^{2s}}$. Moreover, its Gram matrix with respect to $\tr_{E_\P/\lk}(x\overline{y})$ is
\[ G = \begin{pmatrix} 2 & \theta + \overline{\theta} \\ \theta + \overline{\theta} & 2 \, \theta\overline{\theta} \end{pmatrix} = \begin{pmatrix} 2 & 2\beta/\Pi^s \\ 2\beta/\Pi^s &  2(\beta_\P^2 + a)/\Pi^{2s} \end{pmatrix}. \]

We claim that $\L_{0,\lk}$ is $\Pp^{\varepsilon-s}$-modular. Indeed, the determinant of $G$ is $4a/\Pi^{2s}$, whose valuation is $2(\varepsilon-s)$, since $a \notin \P$. In particular the determinant ideal of \(\L_{0,\lk}\) is
\(\P^{2(\varepsilon-s)}\). Moreover, the coefficient of $G$ with the smallest valuation is $2\beta/\Pi^s$, with $v_\P(2\beta/\Pi^s) = \varepsilon - s$, so the scale is $\s(\L_{0,\lk}) = \P^{\varepsilon-s}$. By \cite[82:14]{o2013introduction}, we conclude that $\L_{0,\lk}$ is $\P^{\varepsilon-s}$-modular. 
Additionally, note that the form $\tr_{E_\P/\lk}(x\overline{y})$ is anisotropic on $\O_{E} \otimes \lk \simeq \lk^2$.\footnote{Indeed, the form is isotropic on $\lk^2$ if and only if there exist nonzero elements $x,y \in \lk$ such that $x^2 + a y^2 = 0.$
In other words, this space is isotropic if and only if $-a$ is a square in $\lk$, which is equivalent to saying that $\P$ splits in $E$.} 

$a.$ Assume that $s=\varepsilon$. Then $\L_{0,\lk}$ is unimodular, and according to Remark \ref{rmk}, the ideal $\P$ is inert so $\Qq^v = \P^v \O_{\mathcal{E}}$. In particular,  $\L_{v,\lk}$ is $\P^{2v}$-modular, and by Lemma \ref{lem:dyadic-rank-two}, we obtain 
$$
\L_{v, \lk} \simeq_{\O_{\lk}} \langle \Pi^{2v} \rangle \otimes A(2,2\rho).
$$

$b.$ Finally, assume that $s<\varepsilon$. By Lemma~\ref{lem}, the prime
$\P$ is ramified. We have
$\mathfrak s(\L_{v,\lk})=\P^{\varepsilon-s+v}$, while its determinant
ideal is $\P^{2(\varepsilon-s+v)}$. Hence
$\L_{v,\lk}$ is $\P^{\varepsilon-s+v}$-modular, and Lemma~\ref{lem:dyadic-rank-two}
gives
\[
\L_{v,\lk}\simeq_{\O_\lk}
\langle\Pi^{\varepsilon-s+v}\rangle\otimes A(2,2\rho).
\]
This concludes the proof of the lemma.
\end{demo}

\subsection{Over the totally ramified extension}\label{sec:tot-ram-dyadic}

We now descend from $\lk$ to its maximal unramified subextension $\lm$, as in Section \ref{sec:totram}. Again, the
split case and the inert/ramified case are treated separately. 

\subsubsection{Split case} In the split
case, Gram matrix computation from Lemma \ref{lem:Gram-split} still, but it now leads to decompositions into hyperbolic
planes. 

\begin{proposition}\label{prop:totram-split-dyadic}
Assume that $2\in \Pp$ and that $\Pp$ splits in $E$. Let $v_1,v_2 \in \Z$ and write $qe + r$ for the Euclidean division of $v_1 + v_2$ by $e$. Then $\L_{v_1,v_2,\lm}$ is $\O_\lm$-isometric to one of the following.
\begin{arabiclist}
\item $\langle \pi^q \rangle \otimes \H 
\bigperp
\langle \pi^{q+1} \rangle \otimes \H^{e-1},$ if $r=0;$
\item $\langle \pi^{q+1} \rangle \otimes \H^{e-r+1}
\bigperp
\langle \pi^{q+2} \rangle \otimes \H^{r-1},$ if $r>0$.
\end{arabiclist}
\end{proposition}

\begin{demo}
Recall that a Gram matrix of $\L_{v_1,v_2, \lm}$ is given by Lemma \ref{lem:Gram-split}. Note that $e$ is odd, since $\lk/\lm$ is by assumption (totally) tamely ramified. Let $e_i = (\Pi^{v_1+i},0)$ and $f_j = (0,\Pi^{v_2+j})$ for $0 \le i,j \le e-1$. Then $\mathcal{B} = \{e_i,f_j\}$ is a basis of $\L = \P^{v_1}\times \P^{v_2}$ over $\O_\lm$.

$1.$ As in the proof of Proposition \ref{prop:totram-split-nondyadic}, we set $\L_1 = e_0\O_\lm + f_0\O_\lm$ and $\L_2 = \bigoplus_{i=1}^{e-1}(e_i\O_\lm + f_{e-i}\O_\lm)$. Then $\L$ is the orthogonal summand of $\L_1$ with $\L_2$, whose Gram matrices are $J_q^{(2)}$ and $\mathrm{diag}(J_{q+1}^{(2)},\dots,J_{q+1}^{(2)})$, respectively.
Since $e$ is a unit in $\O_\lm$, we can make the change of variables $f_0' = e^{-1}f_0$. Thus, the Gram matrix of $\L_1$ becomes $\pi^q \begin{psmallmatrix}0 & 1 \\ 1 & 0\end{psmallmatrix}$, so $\L_1 \simeq \langle \pi^q \rangle \otimes \H$. The same argument applies to each summand of $\L_2$.

2. Assume $r > 0$ and let $\L_1 = \bigoplus_{i=0}^{e-r}(e_i\O_\lm + f_{e-r-i}\O_\lm)$, with Gram matrix $\mathrm{diag}(J_{q+1}^{(2)},\dots)$, and $\L_2 = \bigoplus_{i=1}^{r-1}(e_{e-r+i}\O_\lm + f_{e-i}\O_\lm)$ with Gram matrix $\mathrm{diag}(J_{q+2}^{(2)},\dots)$. Then $\L = \L_1 \perp \L_2$ and we conclude as in the previous case.
\end{demo}

\subsubsection{Inert or ramified case} As explained at the beginning of this section, the case $a \in \Pp$ will not be treated here. We therefore
focus on the case $a \notin \Pp$ and the remaining difficulty, namely replacing the unit $\rho \in \O_\lk^\times$
by a unit of $\O_\lm^\times$.

\begin{lemma}\label{lem:change-delta-dyadic}
Let $\rho_\lm \in \O_\lm^\times$ be such that $\idd(1 - 4\rho_\lm) = 4 \O_\lm$ (\textit{c.f.,} Section \ref{subsec:dyadic-prelim}). Through the natural inclusion of $\O_\lm ^\times$ in $\O_\lk^\times$, we still have \[
\idd(1-4\rho_\lm) = 4\O_\lk.
\]
\end{lemma}

\begin{demo}
Note that $\lk$ and $\lm$ have the same residual field, say, $k$, since the extension $\lk/\lm$ is totally ramified. Set $\Delta = 1- 4\rho_\lm$, so that $\idd(\Delta) = 4 \O_\lm$. Following Lemma \ref{lem:charac-Delta}, it suffices to prove that $\lk' = \lk(\sqrt{\Delta})$ is the unique unramified quadratic extension of $\lk$. Note that $\lm' = \lm(\sqrt{\Delta})$ is a subfield of $\lk'$, and that $\lm(\sqrt{\Delta})$ is the unique unramified quadratic extension of $\lm$, by \textit{loc. cit.} In particular, the residual field $k_{\lm'}$ of $\lm'$ is a quadratic extension of $k$. The diagrams of local and residue field extensions are as follows.
\[
\begin{tikzcd}[column sep=0.8em, row sep=0.8em]
  & \lk^{\prime}=\lk(\sqrt{\Delta}) \\
  \lk
    \arrow[ur, no head, "1\text{ or }2"'] & \\
  & \lm^{\prime}=\lm(\sqrt{\Delta})
    \arrow[uu, no head] \\
  \lm
    \arrow[uu, no head, "e"]
    \arrow[ur, no head, "2"'] &
\end{tikzcd}
\qquad\qquad
\begin{tikzcd}[column sep=0.8em, row sep=0.8em]
  & & k_{\lk^{\prime}} \\
  k
    \arrow[urr, no head] & & \\
  & & k_{\lm^{\prime}}
    \arrow[uu, no head] \\
  k
    \arrow[uu, equal]
    \arrow[urr, no head, "2"'] & &
\end{tikzcd}
\]
Therefore, $[k_{\lk'} : k ] = 2 \, [k_{\lk'} : k_{\lm'}] \geq 2$, while on the other hand we have $[k_{\lk'} : k] \leq [\lk' : \lk] \in \{1,2\}$. Thus, $[\lk' : \lk] = 2$ and $[k_{\lk'} : k] = 2$, that is, $\lk'$ is the unique unramified quadratic extension of $\lk$. Applying again Lemma \ref{lem:charac-Delta}, we conclude that $\idd(\Delta) = 4\O_{\lk}$.
\end{demo}

\begin{lemma}\label{lem:gram-dyadic-tot-ram}
Assume that $2 \in \P$ and $a \notin \P$. Let $\varepsilon = v_\P(2)$ and $0 \leq s \leq \varepsilon$ be the integer defined in Lemma~\ref{lem}. We fix a unit $\rho_\lm \in \O_{\lm}^\times$ such that $\idd(1 - 4\rho_\lm) = 4\cdot \O_{\lk}$ (\textit{c.f.,} Lemma~\ref{lem:change-delta-dyadic}). Finally, let $v \in \Z$. Then, $\L_{v,\lm}$ has Gram matrix
\[
\begin{pmatrix}
2 \vec{G}_{q,r} & \vec{G}_{q,r} \\
\vec{G}_{q,r} & 2\rho_\lm \vec{G}_{q,r}
\end{pmatrix},
\]
where the integers $q,r$ are given by
\begin{arabiclist}
\item The Euclidean division $2v = qe + r$ of $2v$ by $e$, when $s = \varepsilon;$

\item The Euclidean division $\varepsilon - s + v = qe + r$ of $\varepsilon - s + v$ by $e$, when $s < \varepsilon$.
\end{arabiclist}
\end{lemma}

\begin{demo}
We treat the case $s=\varepsilon$, the other case being entirely
similar. By Lemma~\ref{lem:change-delta-dyadic}, the unit $\rho_\lm$ satisfies
$\idd(1-4\rho_\lm)=4\O_\lk$. Hence, by
Lemma~\ref{lem:inert-ram-dyadic}, there exists a basis
$\{v,w\}$ of $\L_{v,\lk}$ whose Gram matrix is
$\Pi^{2v}A(2,2\rho_\lm)$.
It follows that the family $
\{v, v\Pi, \dots, v\Pi^{e-1},\, w, w\Pi, \dots, w\Pi^{e-1}\}
$
forms a basis of $\L_{v,\lm}$ over $\O_\lm$, and its Gram matrix is
$
\begin{psmallmatrix}
2 \vec{G}_{q,r} & \vec{G}_{q,r} \\
\vec{G}_{q,r} & 2\rho \vec{G}_{q,r}
\end{psmallmatrix},
$
as claimed.
\end{demo}

\begin{proposition}\label{prop:dyadic-totram-inert}
Assume that $2 \in \P$ and $a \notin \P$. Moreover, let $v,q,r$ and $\rho_\lm$ be as in Lemma \ref{lem:gram-dyadic-tot-ram}. Then, $\L_{v,\lm}$ is $\O_\lm$-isometric to one of the following.
\begin{arabiclist}
\item $\langle e\pi^q \rangle \otimes A(2,2\rho_\lm) \bigperp \langle e\pi^{q+1} \rangle \otimes \H^{e-1},$ \, if $r = 0;$

\item $\langle e\pi^{q+1} \rangle \otimes \big(A(2,2\rho_\lm) \midperp \H^{e-r}\big) \bigperp \langle e\pi^{q+2} \rangle \otimes \H^{r-1},$ \, if $r>0$ is odd$;$

\item $\langle e\pi^{q+1} \rangle \otimes \H^{e-r+1} \bigperp \langle e\pi^{q+2} \rangle \otimes \big(A(2,2\rho_\lm) \midperp \H^{r-2}\big),$ \, if $r>0$ is even.
\end{arabiclist}
Here, a summand involving $\H^0$ is understood to be omitted.
\end{proposition}

\begin{demo}
As in the proof of Lemma \ref{lem:gram-dyadic-tot-ram}, we let $\{v, v\Pi, \dots, v\Pi^{e-1},\, w, w\Pi, \dots, w\Pi^{e-1}\}$ be the basis of $\L_{v,\lm}$ whose Gram matrix is $\begin{psmallmatrix}
2 \vec{G}_{q,r} & \vec{G}_{q,r} \\
\vec{G}_{q,r} & 2\rho_\lm \vec{G}_{q,r}
\end{psmallmatrix}$.

$1.$ Let $\L_0$ be the sublattice of $\L_{v,\lm}$ spanned by $\L_0 = v \O_\lm + w \O_\lm$. Then the Gram matrix of $\L_0$ with respect to $\{v,w\}$ is $e\pi^q \begin{psmallmatrix}
2 & 1 \\
1 & 2\rho_\lm 
\end{psmallmatrix}$, in particular $\L_0$ is $\p_\lm^q$-modular and $\L_0 \simeq_{\O_\lm} \langle e\pi^q \rangle \otimes A(2,2\rho_\lm)$. 

For $1 \leq i \leq \frac{e-1}{2}$,\footnote{Note that $e$ is odd, since the extension is tamely ramified.} we define the rank-$4$ sublattice $\L_i$, spanned by the family $\{v\Pi^i, w\Pi^i, v\Pi^{e-i}, w\Pi^{e-i}\}.$ Then, $\L_i$ and $\L_j$ are orthogonal for any $i \neq j$, and the Gram matrix of $\L_i$ is $\begin{psmallmatrix}
\vec{0} & \vec{\Gamma} \\ \vec{\Gamma} & \vec{0}
\end{psmallmatrix}$, where $\vec{\Gamma} = e\pi^{q+1} \begin{psmallmatrix}
2 & 1 \\ 1 & 2\rho_\lm
\end{psmallmatrix}$. Thus, $\L_i$ is $\p_\lm^{q+1}$-modular and totally isotropic, so it splits into two hyperbolic planes (\cite[93:18]{o2013introduction}), that is, $\L_i \simeq_{\O_\lm} \langle e\pi^{q+1} \rangle \otimes (\H \perp \H)$. This proves the case $r=0$.

Assume that $r>0$. Whenever $i\neq j$ and $i+j=e-r$ (resp. $i+j=2e-r$), the sublattice spanned by $\{v\Pi^i,w\Pi^i,v\Pi^j,w\Pi^j\}$ has Gram matrix $\begin{psmallmatrix}
\vec{0} & \vec{\Gamma} \\ \vec{\Gamma} & \vec{0}
\end{psmallmatrix}$, where $\vec{\Gamma}=e\pi^{q+1}\begin{psmallmatrix}
2&1\\1&2\rho_\lm
\end{psmallmatrix}$ (resp. $\vec{\Gamma}=e\pi^{q+2}\begin{psmallmatrix}
2&1\\1&2\rho_\lm
\end{psmallmatrix}$). As above, this sublattice is isometric to $\langle e\pi^{q+1}\rangle\otimes(\H\perp\H)$ (resp. $\langle e\pi^{q+2}\rangle\otimes(\H\perp\H)$).

$2.$ Suppose first that $r$ is odd. Then $i_0=(e-r)/2$ is the unique index satisfying $2i_0=e-r$. The sublattice $v\Pi^{i_0}\O_\lm+w\Pi^{i_0}\O_\lm$ has Gram matrix $e\pi^{q+1}\begin{psmallmatrix}
2&1\\1&2\rho_\lm
\end{psmallmatrix}$, and is therefore isometric to $\langle e\pi^{q+1}\rangle\otimes A(2,2\rho_\lm)$. The remaining indices satisfying $i+j=e-r$ form pairs and yield $\H^{e-r}$, while those satisfying $i+j=2e-r$ yield the hyperbolic summand $\H^{r-1}$. 

$3.$ Suppose now that $r$ is even. The indices satisfying $i+j=e-r$ form pairs and yield $\H^{e-r+1}$. On the other hand, $i_0=e-r/2$ is the unique index satisfying $2i_0=2e-r$. The corresponding rank-two sublattice is isometric to $\langle e\pi^{q+2}\rangle\otimes A(2,2\rho_\lm)$, while the remaining indices yield $\langle e\pi^{q+2}\rangle\otimes\H^{r-2}$.
\end{demo}

\subsection{Over the unramified extension} In this section, recall that $\lm/\lf$ is a degree $f$ unramified dyadic extension. In light of Proposition \ref{prop:dyadic-totram-inert}, we now need to replace the unit $\rho_\lm \in \O_\lm^\times$ by a suitable unit
$\rho_\lf \in \O_\lf^\times$. The following lemma shows that this is possible only if $f$ is odd.

\begin{lemma}\label{lem:rho-in-F}
Let $\rho_\lf \in \O_\lf^\times$ be such that $\idd(1-4\rho_\lf) = 4 \cdot \O_\lf$ (\textit{c.f.,} Section \ref{subsec:dyadic-prelim}). Seen as an element of $\O_\lm^\times$, we have $\idd(1-4\rho_\lf) = 4 \cdot \O_\lm$ if and only if $f$ is odd.
\end{lemma}

\begin{demo}
Let $\rho_\lf \in \O_\lf^\times$ be a unit such that $\Delta = 1 - 4\rho_\lf$ has quadratic defect equal to $4 \cdot \O_\lf$. By Lemma \ref{lem:charac-Delta}, this is equivalent to say that $\lf' = \lf(\sqrt{\Delta})$ is the unique unramified quadratic extension of $\lf$. We let $\lm' = \lm(\sqrt{\Delta})$ be the compositum of $\lm$ with $\lf'$. Since both $\lm/\lf$ and $\lf'/\lf$ are unramified, $\lm'/\lf$ is an unramified extension. In particular, its degree equals the degree of the residual extension
\[
[k_{\lm'} : k_{\lf}] = \text{lcm}\Big([k_\lm:k_\lf], [k_{\lf'}, k_\lf]\Big) = \text{lcm}(f,2).
\]
This relation yields the degree of $\lm'/\lm$:
\[
[\lm':\lm] = \frac{[\lm':\lf]}{[\lm:\lf]} = \frac{\text{lcm}(f,2)}{f} = \frac{2}{\gcd(f,2)}.
\]
Therefore, $\lm'/\lm$ is a quadratic (and unramified) extension if and only if $f$ is odd. We conclude again by using Lemma \ref{lem:charac-Delta}.
\end{demo}

We now proceed as in the non-dyadic case, using the analogue of Lemma~\ref{prop:interm1} to descend a splitting over $\O_\lm$ to a splitting
over $\O_\lf$.

\begin{lemma}\label{lem:interm-dyadic}
Assume $2 \in \Pp$, and let $(\L,q)$ be a quadratic
$\O_\lm$-lattice. Suppose that there exist $r \in \Z_{\geq 0}$, $\alpha, \beta \in \O_\lf$ and $\gamma \in \O_\lm$ such that
\[
(\L,q) \simeq_{\O_\lm}
\langle \gamma \rangle \otimes A(\alpha,\beta)^r.
\]
Then, over $\O_\lf$ one has
\[
(\L,\tr_{\lm/\lf}\circ \, q)
\simeq_{\O_\lf}
\bigperp_{i=1}^{r}
A(\alpha,\beta)\otimes(\O_\lm,q^{\gamma}_{\lm/\lf}).
\]
\end{lemma}

\begin{demo}
By assumption, there is a basis $\{\vec v_1,\vec v_1',\dots,\vec v_r,\vec v_r'\}$ of $\L$ over $\O_\lm$ such that $(\L,q)$ is the orthogonal sum of the
rank-two lattices
\[
(V_i,q):=(\vec v_i\O_\lm+\vec v_i'\O_\lm,q)
\simeq_{\O_\lm}
\langle \gamma \rangle\otimes A(\alpha,\beta),
\]
for $1\leq i\leq r$. For each $i$, the above
isometry means that $q(\vec v_i,\vec v_i)=\gamma \alpha, \, 
q(\vec v_i,\vec v_i')=\gamma,$ and $q(\vec v_i',\vec v_i')=\gamma \beta.$ Let $\{\vec b_1,\dots,\vec b_f\}$ be an $\O_\lf$-basis of $\O_\lm$, and fix
$1\leq i\leq r$. Then, $(\vec e_1,\dots,\vec e_{2f})
:=
(\vec b_1\vec v_i,\dots,\vec b_f\vec v_i,
 \vec b_1\vec v_i',\dots,\vec b_f\vec v_i')$ 
is an $\O_\lf$-basis of $V_i$. For $1\leq j,k\leq f$, we have
\[
\left\{
\begin{array}{ll}
\tr_{\lm/\lf}\bigl(q(\vec e_j,\vec e_k)\bigr)
=
\tr_{\lm/\lf}(\gamma\alpha \vec b_j\vec b_k)
=
\alpha\,\tr_{\lm/\lf}(\gamma\vec b_j\vec b_k), \\[0.2cm]

\tr_{\lm/\lf}\bigl(q(\vec e_j,\vec e_{k+f})\bigr)
=
\tr_{\lm/\lf}(\gamma\vec b_j\vec b_k), \\[0.2cm]

\tr_{\lm/\lf}\bigl(q(\vec e_{j+f},\vec e_k)\bigr)
=
\tr_{\lm/\lf}(\gamma\vec b_j\vec b_k), \\[0.2cm]

\tr_{\lm/\lf}\bigl(q(\vec e_{j+f},\vec e_{k+f})\bigr)
=
\tr_{\lm/\lf}(\gamma\beta \vec b_j\vec b_k)
=
\beta\,\tr_{\lm/\lf}(\gamma\vec b_j\vec b_k).
\end{array}
\right.
\]
Here we used that $\alpha, \beta\in\O_\lf$, so they can be pulled out of the trace. Since $\tr_{\lm/\lf}$ preserves the orthogonality of the $V_i$'s, we
obtain the decomposition claimed.
\end{demo}

Next we prove a slight generalization of \cite[Lemma 1.11]{mantillagenus} to
arbitrary subfields, not necessarily $\Q_2$. This will be useful when combining
the previous lemma with the decompositions computed in
Section~\ref{sec:tot-ram-dyadic}.

\begin{lemma}\label{lem:hyperbolic-tensor-trace}
Let $(\L,q)$ be an unimodular $\O_\lf$-lattice of rank $r$. Then,
\[ \mathbb{H} \otimes (\L,q) \simeq_{\O_\lf} \mathbb{H}^r. \]
\end{lemma}

\begin{demo}
By \cite[93:15]{o2013introduction}, the unimodular lattice $(\L,q)$ is an
orthogonal sum of lattices of the following types: $\langle u \rangle$ with $u \in \O_\lf^\times$, hyperbolic planes
$\mathbb{H}$, or
$A(\alpha,\beta)$ with $(\alpha, \beta)\neq (0,0)$.
Therefore, it is enough to prove that
tensoring each of these building blocks with $\mathbb{H}$ yields an orthogonal
sum of hyperbolic planes.

By definition, 
$\mathbb{H}\otimes \langle u\rangle$ has a basis $\{\vec{v}, \vec{w}\}$ whose Gram matrix is $\begin{psmallmatrix}
0 & u \\ u & 0
\end{psmallmatrix}.$ The change of variables $\vec{v'} = u^{-1} \vec{v}$ yields the result. Now assume that $\L$ is a rank-two unimodular lattice with Gram matrix $G\in \GL_2(\O_\lf)$.
This covers both cases $\L=\mathbb{H}$ and $\L=A(\alpha,\beta)$. Then, $\mathbb{H}\otimes \L$ has Gram matrix
\[
\begin{pmatrix}
0&G\\
G&0
\end{pmatrix}.
\]
After the change of basis given by $T=
\begin{psmallmatrix}
I_2&0\\
0&G^{-1}
\end{psmallmatrix},$
the Gram matrix becomes
\[
{}^tT
\begin{pmatrix}
0&G\\
G&0
\end{pmatrix}
T
=
\begin{pmatrix}
0&I_2\\
I_2&0
\end{pmatrix},
\]
which is the Gram matrix of $\mathbb{H}^2$. Hence we obtain $\mathbb{H}\otimes \L \simeq_{\O_\lf} \mathbb{H}^2$, as claimed.
\end{demo}

\subsection{Gluing the decomposition} 
Similarly to Section \ref{sec:gluing-non-dyadic}, we fix a unit $\nu_\lm \in \O_\lm^\times$ such that $\pi_\lf := \nu_\lm \pi$ is a uniformizer of $\lf$, and we also set $\nu_\lf := \norm_{\lm/\lf}(\nu_\lm)$.

\subsubsection{Split case}

\begin{proposition}\label{glue-dyadic-split}
Assume that $2 \in \Pp$ and that $\Pp$ splits in $E$. Let $v_1,v_2 \in \Z$ and
write $v_1+v_2=qe+r$ for the Euclidean division of $v_1+v_2$ by $e$. Then, $\L_{v_1,v_2,\lf}$ is
$\O_\lf$-isometric to one of the following.
\begin{arabiclist}
    \item $\langle \pi_\lf^q\rangle \otimes \mathbb{H}^f
    \;\bigperp\;
    \langle \pi_\lf^{q+1}\rangle \otimes \mathbb{H}^{f(e-1)}$,
    if $r=0;$

    \item $\langle \pi_\lf^{q+1}\rangle \otimes \mathbb{H}^{f(e-r+1)}
    \;\bigperp\;
    \langle \pi_\lf^{q+2}\rangle \otimes \mathbb{H}^{f(r-1)}$,
    if $r>0$.
\end{arabiclist}
\end{proposition}

\begin{demo}
We only treat the case $r=0$, the case $r>0$ being entirely similar. By
Proposition~\ref{prop:totram-split-dyadic}, one has the splitting $\langle \pi^q\rangle \otimes \mathbb{H}
\bigperp
\langle \pi^{q+1}\rangle \otimes \mathbb{H}^{e-1}$ over $\O_\lm.$
Applying Lemma~\ref{lem:interm-dyadic} to each hyperbolic summand, we obtain $\L_{v_1,v_2,\lf}\simeq_{\O_\lf}$
\[
\langle \pi_\lf^q \rangle \otimes \mathbb{H}\otimes \Big(\O_\lm,q^{{\nu_\lm}^{-q}}_{\lm/\lf}\Big)
\bigperp \left(
\langle \pi_\lf^{q+1} \rangle \otimes\bigperp_{i=1}^{e-1}\mathbb{H}\otimes \Big(\O_\lm,q^{{\nu_\lm}^{-(q+1)}}_{\lm/\lf}\Big) \right).
\]
Now Lemma~\ref{lem:hyperbolic-tensor-trace} shows that
\(\mathbb{H}\otimes (\O_\lm,q^{\nu_\lm^s}_{\lm/\lf})
\simeq_{\O_\lf} \mathbb{H}^f\)
for every \(s\in\Z\). Therefore,
\[
\L_{v_1,v_2,\lf}\simeq_{\O_\lf}
\langle \pi_\lf^q\rangle \otimes \mathbb{H}^f
\bigperp
\langle \pi_\lf^{q+1}\rangle \otimes \mathbb{H}^{f(e-1)},
\]
which is exactly the claimed decomposition when \(r=0\).
\end{demo}

\subsubsection{Inert or ramified case}

\begin{proposition}\label{glue-dyadic-inert-ramified}
Assume that $2 \in \Pp$, that $a \notin \Pp$, and let $v \in \Z$. Let
$\varepsilon:=v_\Pp(2)$ and let $0 \leq s \leq \varepsilon$ be as in
Lemma~\ref{lem}. Moreover, let $q,r$ be the quotient and remainder in the
Euclidean division of $2v$ by $e$, if $s=\varepsilon$ (resp. of
$\varepsilon-s+v$ by $e$, if $s<\varepsilon$). Assume furthermore that $f$
is odd, and fix a unit $\rho_\lf \in \O_\lf^\times$ such that
$\idd(1-4\rho_\lf)=4\O_\lf$ (\textit{c.f.}, Lemma~\ref{lem:rho-in-F}).

Then, $\L_{v,\lf}$ is $\O_\lf$-isometric to one of the following.
\begin{arabiclist}
    \item If $r=0$, then
    \[
    \left( \langle \pi_\lf^q\rangle \otimes
    A(2,2\rho_\lf)\otimes
    \Big(\O_\lm,q^{e\nu_\lm^{-q}}_{\lm/\lf}\Big)
    \right)
    \bigperp
    \langle \pi_\lf^{q+1}\rangle \otimes \mathbb{H}^{f(e-1)}.
    \]

    \item If $r>0$ is odd, then
    \[
    \langle \pi_\lf^{q+1}\rangle \otimes
    \left(
    \left(
    A(2,2\rho_\lf)\otimes
    \Big(\O_\lm,q^{e\nu_\lm^{-(q+1)}}_{\lm/\lf}\Big)
    \right)
    \bigperp
    \mathbb{H}^{f(e-r)}
    \right)
    \bigperp
    \langle \pi_\lf^{q+2}\rangle \otimes
    \mathbb{H}^{f(r-1)}.
    \]

    \item If $r>0$ is even, then
    \[
    \langle \pi_\lf^{q+1}\rangle \otimes
    \mathbb{H}^{f(e-r+1)}
    \bigperp
    \langle \pi_\lf^{q+2}\rangle \otimes
    \left(
    \left(
    A(2,2\rho_\lf)\otimes
    \Big(\O_\lm,q^{e\nu_\lm^{-(q+2)}}_{\lm/\lf}\Big)
    \right)
    \bigperp
    \mathbb{H}^{f(r-2)}
    \right).
    \]
\end{arabiclist}
Here, a summand involving $\mathbb H^0$ is understood to be omitted.
\end{proposition}

\begin{demo}
Since $f$ is odd, Lemma~\ref{lem:rho-in-F} shows that the unit $\rho_\lf$
still satisfies $\idd(1-4\rho_\lf)=4\O_\lm$ when viewed in
$\O_\lm^\times$. Hence, in Proposition~\ref{prop:dyadic-totram-inert}, we
may take $\rho_\lm=\rho_\lf$.

Recall that $\pi_\lf=\nu_\lm\pi$. For any $t\in\Z$, the anisotropic block
$\langle e\pi^t\rangle\otimes A(2,2\rho_\lf)$ may be written as
$\langle\pi_\lf^t\rangle\otimes
\big(\langle e\nu_\lm^{-t}\rangle\otimes A(2,2\rho_\lf)\big)$.
Applying Lemma~\ref{lem:interm-dyadic} with
$\gamma=e\nu_\lm^{-t}$, $\alpha=2$ and $\beta=2\rho_\lf$, it descends to
\[
\langle \pi_\lf^t\rangle \otimes
\left(
A(2,2\rho_\lf)\otimes
\Big(\O_\lm,q^{e\nu_\lm^{-t}}_{\lm/\lf}\Big)
\right).
\]

Similarly, a hyperbolic block
$\langle e\pi^t\rangle\otimes\mathbb H$ descends, by
Lemma~\ref{lem:interm-dyadic} and
Lemma \ref{lem:hyperbolic-tensor-trace}, to
$\langle\pi_\lf^t\rangle\otimes\mathbb H^f$. Indeed,
$e\nu_\lm^{-t}$ is a unit and
$\big(\O_\lm,q^{e\nu_\lm^{-t}}_{\lm/\lf}\big)$ is therefore an
unimodular $\O_\lf$-lattice of rank $f$. We now distinguish the three cases.

$1.$ If $r=0$, Proposition~\ref{prop:dyadic-totram-inert} gives one
anisotropic block at scale $\p_\lm^q$ and $e-1$ hyperbolic planes at
scale $\p_\lm^{q+1}$. Applying the preceding observations gives the result in this case.

$2.$ If $r>0$ is odd, the anisotropic block occurs at scale
$\p_\lm^{q+1}$, together with $e-r$ hyperbolic planes at the same scale
and $r-1$ hyperbolic planes at scale $\p_\lm^{q+2}$. This gives the
second decomposition. 

$3.$ Finally, if $r>0$ is even, there are $e-r+1$
hyperbolic planes at scale $\p_\lm^{q+1}$, while the anisotropic block
and $r-2$ hyperbolic planes occur at scale $\p_\lm^{q+2}$. This gives
the last decomposition.
\end{demo}

\begin{remark}
How explicit can the remaining non-hyperbolic block be made?
In each of the above decompositions, it has the form
\[
\langle \pi_\lf^t\rangle \otimes
\left(
A(2,2\rho_\lf)\otimes
(\O_\lm,q^{u_\lm}_{\lm/\lf})
\right),
\]
for some $t\in\Z$ and $u_\lm\in\O_\lm^\times$. If $u_\lm$ is a square,
then $(\O_\lm,q^{u_\lm}_{\lm/\lf})\simeq_{\O_\lf}
(\O_\lm,q_{\lm/\lf})$, and the latter can be described explicitly from
its dyadic invariants. Indeed, it has weight $2\O_\lf$. Since $f$ is odd,
set $\delta=(-1)^{(f-1)/2}d_{\lm/\lf}\in\O_\lf^\times$. If $f\geq 3$,
\cite[Proposition~3.3.11(1)]{kirschmer2016definite} gives a unit
$v_\p\in\O_\lf^\times$ such that
\[
(\O_\lm,q_{\lm/\lf})
\simeq_{\O_\lf}
\langle\delta\rangle
\perp A(2v_\p,0)
\perp \H^{(f-3)/2}.
\]
If \(f=1\), then \(\lm=\lf\) and $(\O_\lm,q_{\lm/\lf})\simeq_{\O_\lf}\langle1\rangle,$
so the factor is \(A(2,2\rho_\lf)\). 

Thus, whenever $u_\lm$ is a square, the remaining non-hyperbolic
component contains an explicit anisotropic rank-two factor.
\end{remark}

\subsubsection{Jordan decomposition in the dyadic case}

\begin{theorem}\label{thm:dyadic-jordan}
Assume that $\p$ is dyadic and that, for every prime ideal
$\Pp\mid\p$ which is not split in $E$, one has $a\notin\Pp$ and the
residue degree $f_\Pp=[\O_K/\Pp:\O_F/\p]$
is odd. Let $\ida\subset E$ be a non-zero fractional ideal, and set
$\ida_\Pp:=\ida\otimes_{\O_K}\O_{K_\Pp}$ for every $\Pp\mid\p$. Then, locally at $\p$,
\[
\L_{\ida,F}\simeq_{\O_{F_\p}}
\bigperp_{\Pp\mid\p}\L_{\ida_\Pp,F_\p}.
\]
If $\Pp$ splits in $E$ and
$\ida_\Pp=\Qq_1^{v_1}\Qq_2^{v_2}$, the corresponding factor is described
by Proposition~\ref{glue-dyadic-split}. If $\Pp$ is inert or ramified
and $\ida_\Pp=\Qq^v$, it is described by
Proposition~\ref{glue-dyadic-inert-ramified}.

Moreover, the Jordan decomposition of $\L_{\ida,F}$ at $\p$ is obtained by grouping
together the modular blocks having the same scale.
\end{theorem}

\begin{demo}
The result follows from Lemma~\ref{lemma:orthoplit} and
Propositions~\ref{glue-dyadic-split} and
\ref{glue-dyadic-inert-ramified}, exactly as in the proof of
Theorem~\ref{thm:jordan-decomp-nondyadic}.
\end{demo}

Under the hypotheses of Theorem~\ref{thm:dyadic-jordan}, the local factor
attached to a prime $\Pp\mid\p$ is determined by its splitting type in $E$
and by the following integer. Set
\[
u_\Pp(\ida):=
\begin{cases}
v_{\Qq_1}(\ida)+v_{\Qq_2}(\ida),
& \text{if } \Pp\O_E=\Qq_1\Qq_2 \text{ splits in }E;\\[1ex]
2v_{\Qq}(\ida),
& \text{if $\Pp$ is non-split and $s_\Pp=\varepsilon_\Pp$};\\[1ex]
\varepsilon_\Pp-s_\Pp+v_{\Qq}(\ida),
& \text{if $\Pp$ is non-split and $s_\Pp<\varepsilon_\Pp$}.
\end{cases}
\]
where $\Qq$ denotes the unique prime of $E$ above a non-split prime
$\Pp$, and $\varepsilon_\Pp=v_\Pp(2)$, and $s_\Pp$ is defined in
Lemma~\ref{lem}.

\begin{corollary}[Dyadic criterion for $\p$-separated ideals]
\label{cor:dyadic-p-separated}
Assume that $\p$ is dyadic, that $K/F$ is Galois, and that the fractional
ideals $\ida$ and $\idb$ are $\p$-separated. Let $f$ be the common residue
degree of the primes above $\p$, and assume that $f$ is odd and that
$a\notin\Pp$ for every non-split prime $\Pp\mid\p$.

Then, $\L_{\ida,F}$ and $\L_{\idb,F}$ are locally isometric at $\p$
if and only if the multisets
\[
\bigl\{u_\Pp(\ida)\bigr\}_{\substack{\Pp\mid\p\\
\Pp\text{ split in }E}}
\qquad\text{and}\qquad
\bigl\{u_\Pp(\ida)\bigr\}_{\substack{\Pp\mid\p\\
\Pp\text{ non-split in }E}}
\]
coincide respectively with the corresponding multisets for $\idb$.
\end{corollary}

\begin{demo}
As in the proof of Corollary~\ref{cor:local-criterion-nondyadic},
$\p$-separation prevents modular blocks arising from distinct primes above
$\p$ from merging. Since $K/F$ is Galois, all the corresponding local
extensions have the same ramification index and residue degree.
Propositions~\ref{glue-dyadic-split} and
\ref{glue-dyadic-inert-ramified} then show that, for a fixed splitting
type, the isometry class of the $\Pp$-factor depends only on
$u_\Pp(\ida)$.

It remains to distinguish the two splitting types. A split factor is an
orthogonal sum of hyperbolic planes, whereas a non-split factor has one
non-hyperbolic modular component. Since $f$ is odd, there exists a unit $u_\lm\in\O_\lm^\times$ such that this non-hyperbolic
component has the form
\[
A(2,2\rho_\lf)\otimes
\bigl(\O_\lm,q^{u_\lm}_{\lm/\lf}\bigr).
\]
Moreover, its discriminant differs from that of a hyperbolic lattice of the same rank by the
non-square class $\Delta = 1-4\rho_\lf$. Thus, split and non-split factors cannot be
matched. The result follows.
\end{demo}

\section*{Acknowledgements}
This work was
supported by the France 2030 program, managed by the French National Research
Agency under grant agreement No. ANR-22-PETQ-0008 PQTLS. The author thank Alice Pellet-Mary and Renaud Coulangeon for helpful disscusions and comments on the draft.

\bibliographystyle{plain}
\bibliography{ref}

\end{document}